\documentclass[12pt]{article}

\usepackage{a4,epsfig}
\usepackage{amsfonts,amsmath,amsthm}
\usepackage{subcaption}
\usepackage{algorithm}
\usepackage{algpseudocode}
\usepackage{pgfplotstable}
\usepackage{hyperref}
\usepackage{algorithm}
\usepackage{algpseudocode}
\usepackage{xurl}
\usepackage{svg}
\usepackage{adjustbox}
\usepackage{pgfplots}
\usepackage{tikz}
\usepackage{framed} 
\usepackage{subcaption}
\usepackage{pdfpages}
\usepackage{comment}
\usepackage{float} 
\usepackage{tabularx} 
\usepackage{booktabs} 
\usepackage{siunitx} 
\usepackage{makecell}
\let\citep\cite

\newcommand{\cur}{\mathrm{curl}}
\newcommand{\id}{\mathrm{d}}
\newcommand{\dx}{\ \mathrm{d}x}

\newcommand{\J}{\mathcal{J}}
\newcommand{\mH}{\mathcal{H}}
\newcommand{\mL}{\mathcal{L}}
\newcommand{\E}{\mathcal{E}}
\newcommand{\T}{\mathcal{T}}
\newcommand{\R}{\mathbb{R}}

\theoremstyle{plain}
\newtheorem{Theorem}{Theorem}
\newtheorem{Lemma}{Lemma}
\newtheorem{corollary}{Corollary}
\theoremstyle{definition}
\newtheorem{definition}{Definition}
\newtheorem{assumption}{Assumption}
\theoremstyle{remark}
\newtheorem{remark}{Remark}

\begin{document}

\setcounter{page}{1}

\title{Robust Topology Optimization of Electric Machines using Topological Derivatives}
\author{Peter~Gangl$^1$         \and
        Theodor~Komann$^2$  \and
        Nepomuk~Krenn$^1$ \and
Stefan~Ulbrich$^2$
}
\date{$^1$RICAM, Austrian Academy of Sciences, \\
Altenberger Stra{\ss}e 69, 4040 Linz, Austria\\
$^2$Department of Mathematics, TU Darmstadt, \\
Dolivostraße~15, 64297 Darmstadt, Germany}

\maketitle
\begin{abstract}
Designing high-performance electric machines that maintain their efficiency and reliability under uncertain material and operating conditions is crucial for industrial applications. In this paper, we present a novel framework for robust topology optimization with partial differential equation  constraints to address this challenge. The robust optimization problem is formulated as a min-max optimization problem, where the inner maximization is the worst case with respect to predefined uncertainties, while the outer minimization aims to find an optimal topology that remains robust under these uncertainties using the topological derivative. The shape of the domain is  represented by a level set function, which allows for arbitrary perturbation of the domain. The robust optimization problem is solved using a theorem of Clarke  to compute subgradients of the worst case function. This allows the min-max problem to be solved efficiently and ensures that we find a design that performs well even in the presence of uncertainty. Finally, numerical results for a two-material permanent magnet synchronous machine demonstrate both the effectiveness of the method and the improved performance of robust designs under uncertain conditions.

\end{abstract}
\section{Introduction}
\label{Introduction}
Electric machines play an important role in modern industrial applications, including automotive, aerospace, and renewable energy systems \citep{El-Refaie2017}. In light of recent regulatory developments—such as the European Union’s approval of a 2035 phaseout of CO$_2$-emitting vehicles—these machines are expected to become even more important in achieving the transition towards sustainable mobility. Consequently, various factors such as efficiency, performance and manufacturability must be carefully considered when designing these machines. In this context, topology optimization with the consideration of constraints is a popular and  useful technique for designing electric machines with the goal of finding optimal material distribution, leading to improved performance and reduced material costs \citep{Nishanth2022,Gangl_2017ab,Kuci_2018aa,Sanogo_2018aa}.
Topology optimization methodologies can be broadly split into two types: genetic and gradient-based. Genetic optimization algorithms have been extensively applied to topology optimization of  electric machines due to their ability to explore complex design spaces without requiring derivative information   \citep{Sato2015,Bramerdorfer2018}. An advantage of genetic algorithms is the ability to bypass local minima, thereby increasing the likelihood of finding global minima. This is especially useful for topology optimization, where many locally minimal designs are common. However, these methods are black-box algorithms, have a small and by parametrization restricted design space  and can require prohibitively long computation times, which can take days or even weeks. Since our application involves an electric machine with an advanced design space because the physical system is described by a PDE, we will use a derivative-based algorithm in this work. Such an approach can use sensitivity information, allowing us to efficiently explore the design space and avoid the extensive computational costs associated with genetic methods. 
%Yet, their high computational cost and by parametrization restricted design space often make them impractical for complex machine geometries. In contrast, gradient-based methods use sensitivity information to efficiently explore the design space, enabling faster convergence and reduced computational cost. Given the complexity of electric machine topologies, we use a gradient-based method in this work to exploit these advantages.

Derivative-based topology optimization methods can be classified into density-based and level set approaches. Density-based methods represent the design by a density variable taking values between 0 and 1 and interpolate material properties accordingly. %\citep{Gangl_2012aa,Allaire_2005ab,Bendsoe_2013aa,Sigmund_2013aa,Theodore2023, Dede_2012aa,Houta2024}.
The density-based topology optimization in the field of electromagnetism was introduced by Dyck and Lowther in 1996 \citep{Lowther1996}, building on earlier foundational work by Bendsøe in 1989, who described topology optimization as a problem of distributing materials first \citep{Bendsøe1989}. One prominent approach is the \textit{Solid Isotropic Material with Penalization }(SIMP) method, where intermediate material densities are penalized due to there nonphysical behavior. Recently in \citep{Theodore2022} this was successfully applied in multi-material topology optimization using Wachspress interpolation techniques for the design of a 3-phase electrical machine, addressing challenges such as numerical instabilities and sensitivity filtering. Moreover, the density method has been effectively used to optimize electric machines in \citep{Houta2024, Kuci_2018aa}. It is well established and can lead to efficient electromagnetic designs, but choosing an appropriate interpolation function can be challenging, especially when dealing with nonlinear or multi physical problems \citep{Sanogo_2018aa}.
In contrast level set methods represent the material distribution and topology using a level set function that implicitly defines different material phases and thus divides the design domain into separate subdomains.
In topology optimization, the evolution of this level set function can be achieved by computing sensitivities, such as shape derivatives or shape gradients, which are inserted into the Hamilton-Jacobi equation  \citep{Allaire_2004aa} or by the topological derivative \citep{Gangl_2016aa,Amstutz_2006ab}. We will use the level set approach with the topological derivative in this work, as it offers great flexibility for changing the initial design and is well-suited for electric machines due to its ability to precisely define material interfaces and explore complex geometries.

While all these approaches have improved electric machine design, a major gap remains: they often fail to address robustness with respect to manufacturing tolerances and parameter uncertainties. Such uncertainties are important in practice, since even small geometric deviations or variations in material properties can  worsen the performance and cause a design that appears optimal in theory to perform poorly in reality \citep{Lei2021,Ben-Tal_2002aa}. There are two main approaches to handle these uncertainties. The first is to consider a stochastic optimization problem, which assumes that uncertainty can be described probabilistically. This approach  minimizes the expected value of an objective function, as discussed in \citep{Shapiro2009,birge2011introduction} and for optimization problems constrained by a partial differential equation (PDE), in \citep{Kouri2016,Geiersbach2024}. The second, more conservative approach is robust optimization, introduced by Ben-Tal and Nemirovski \citep{Ben-Tal_2002aa}, which  minimizes the worst case objective over a predefined uncertainty set. This deterministic approach has the advantage of ensuring good performance even under the most unfavorable realization of uncertainties. In some engineering applications, ensuring such worst case performance is more desirable than only being robust in terms of the expected value.
The uncertainty affecting our model comes from two main sources: Imperfections in achieving the desired design variables, commonly referred to as \textit{implementation errors} $\delta x \in U_x$, and \textit{parameter uncertainties} $q \in U_q$, which result from inaccuracies in the problem formulation, such as noise \citep{Bertsimas_2010aa}. These sets are combined into a Cartesian product to form the overall uncertainty set: $U = U_{q} \times U_{x}$, which is closed, convex, and bounded. If \( x \) is our design variable, the robust optimization problem can be formulated as:
\begin{equation}
\min_{x} \max_{(\delta x, q) \in U} g(x + \delta x, q),
\label{eq:WorstCase1}
\end{equation}
Here, $g(x + \delta x, q)$ typically measures the performance of the quantity of interest (e.g., average torque or torque ripple in electric machines) for a given design $x$ perturbed by implementation errors $\delta x$ and parameter uncertainties $q$. Therefore the worst case function can be written as 
\begin{equation}
    f(x) = \max_{(\delta x, q) \in U} g(x + \delta x, q)
\end{equation}
Methods to solve this problem will be presented in Section~\ref{sec:robust optimization}. Figure~\ref{fig:RobustOptimizationPlot} illustrates the worst case function and the difference between a robust and a global minimum, with the left subfigure showing uncertainty in the parameter $q$, and the right subfigure showing on uncertainty in the design $\delta x$.  This robust perspective is crucial for electric machines, as ignoring uncertainties in parameter or shape can yield designs that are highly sensitive to small perturbations\citep{Alla_2019aa, Kolvenbach_2018aa, Lass_2017aa, Komann2024, Bontinck2019}. These uncertainties arise from variations in physical parameters and geometric shapes. To address uncertainties in material properties,  \citep{Putek_2016aa} applied generalized polynomial chaos expansions. This method results in electric machine designs that maintain high performance even when material properties change.
Similar observations have been made in other engineering applications. Consequently, robust optimization techniques have been applied to other topology optimization problems. In  density based topology optimization, such as for airwing structures, incorporating uncertainties into the optimization problem has resulted in designs that are not only more robust but also more practical for real-world use \citep{Silva2020}. Another example for topology optimization under uncertainty can be found in  \citep{Conti_2011aa}, who proposed a method for risk-averse shape optimization of elastic structures with uncertain loads. Their approach uses different scenarios to calculate the likelihood of design failure and applies a smooth gradient method to change the shape, allowing new structures to form. \citep{Martinez_Frutos_2016aa} used stochastic collocation together with a level set method for robust shape optimization in linear elasticity to handle uncertainties in loading and material properties. Similarly, \citep{Wu_2016aa} introduced a non-probabilistic robust topology optimization technique using interval uncertainty, where they estimate the worst case with interval arithmetic. In this paper, we present a novel robust topology optimization algorithm that combines a robust optimization technique developed in \citep{Komann2024}, which is based on \citep{Bertsimas_2010aa}, with a level set algorithm that applies the topological derivative \citep{krenn2024}. This combination allows us to efficiently solve robust topology optimization problems formulated as min-max problems, where inner maximization calculates the worst case with respect to predefined uncertainties, while the outer minimization uses the topological derivative to minimize the worst case. We illustrate our technique in a concrete engineering problem by applying it to the robust design optimization of an electric machine.

The remainder of this paper is organized as follows.
Section~\ref{sec:Physical Model} presents the physical model of the robust topology optimization problem.  Section~\ref{sec: Methodology} describes the main algorithmic steps and mathematical theory of both topology optimization and robust optimization. Section~\ref{sec: Combined Approach} combines these ideas into a unified approach to robust topology optimization using the topological derivative. Section~\ref{sec: Numerical Results} provides numerical results for a two-material PMSM, highlighting the advantages of our method. Finally, Section~\ref{sec:Conclusion} concludes the paper and discusses potential directions for future research.
\begin{figure}
    \centering
    \includegraphics[width=0.45\textwidth]{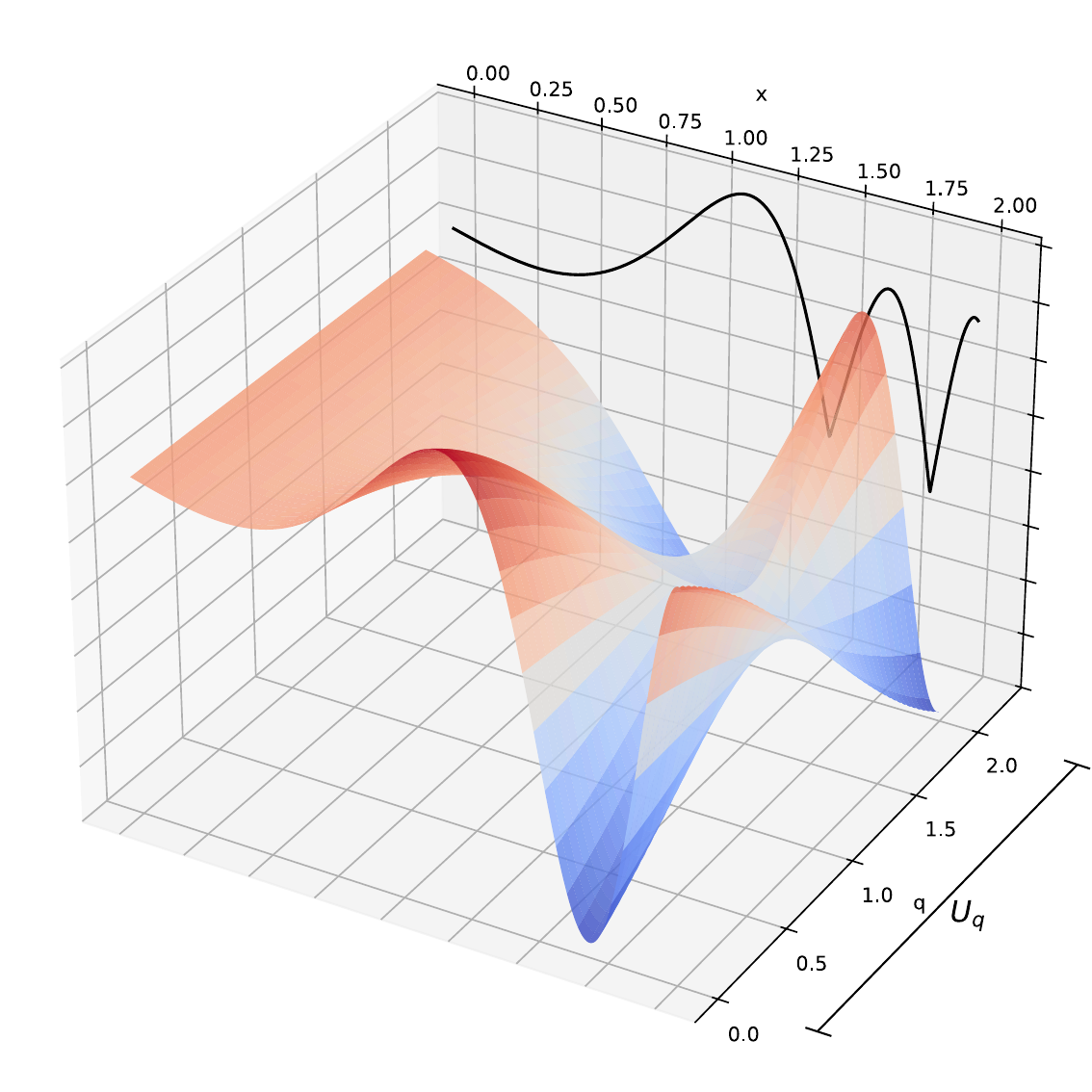}
    \includegraphics[width=0.45\textwidth]{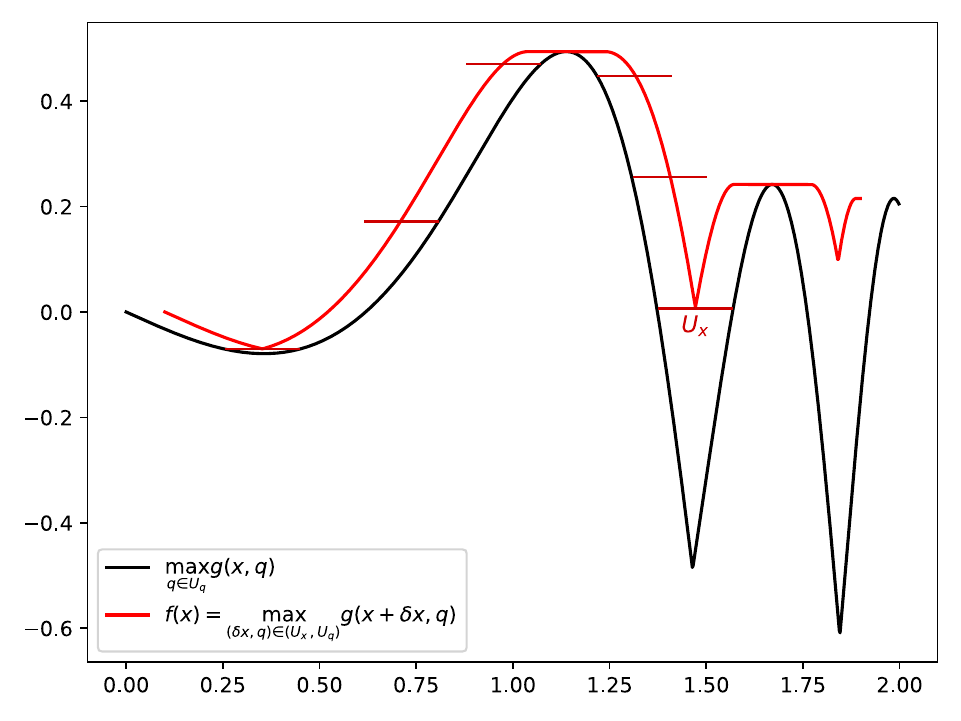}
    \caption{Example of a function $g(x,q)$ with worst case $\max_{q\in U_q}$ in black (left); worst case w.r.t. $q$ (black) and w.r.t $(\delta x,q)$, i.e. $f(x)=\max_{(\delta x,q)\in (U_x,U_q)}g(x+\delta x,q)$ (right).}
    \label{fig:RobustOptimizationPlot}
\end{figure}
\subsection{Notation}\label{sec:notation}
In this paper we will use standard notation for function spaces $C(D), C^\infty_c(D),$ $L_2(D), H^1(D)$ where $D \subset \mathbb{R}^2$ is an open domain. These space refer to continuous, smooth and compactly supported, square integrable and functions with square integrable weak first derivatives, respectively. Further we denote by $\cur v=(\partial_yv,-\partial_xv)^T$ the scalar-to-vector curl differential operator, $\widetilde{\cur}=-\partial_yv_x+\partial_xv_y$ is the vector-to-scalar curl and  $e_\varphi=(\cos\varphi,\sin\varphi)^T\in\R^2$ the unit vector with angle $\varphi$ from the positive $x$-axis. We use the open Euclidean unit ball $B_\epsilon(z):=\{x\in\R^2:\|x-z\|_2<\epsilon\}$ and denote by $\chi_\Omega$ the characteristic function of $\Omega$. By $v|_D$ we denote the restriction of the function $v$ to the domain $D$. The function $\rho_\alpha:\R²\rightarrow\R², \rho_\alpha(x,y)=R_\alpha(x,y)^T$ is the rotational coordinate transformation with the rotation matrix $R_\alpha=(e_\alpha,e_{\alpha+\frac{\pi}{2}}).$

\section{Physical Model}
\label{sec:Physical Model}
\begin{figure}
    \centering
    \includegraphics[width=0.48\textwidth]{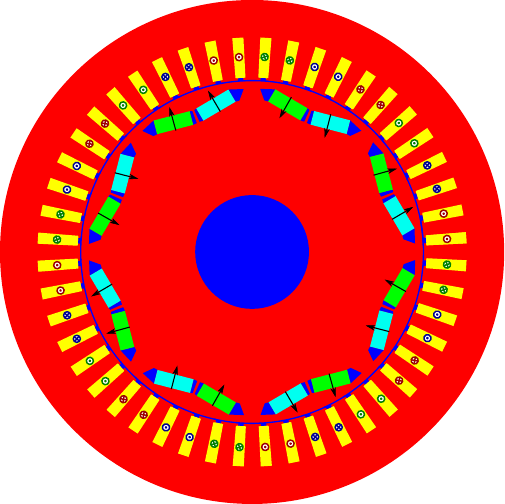}
    \includegraphics[width=0.48\textwidth]{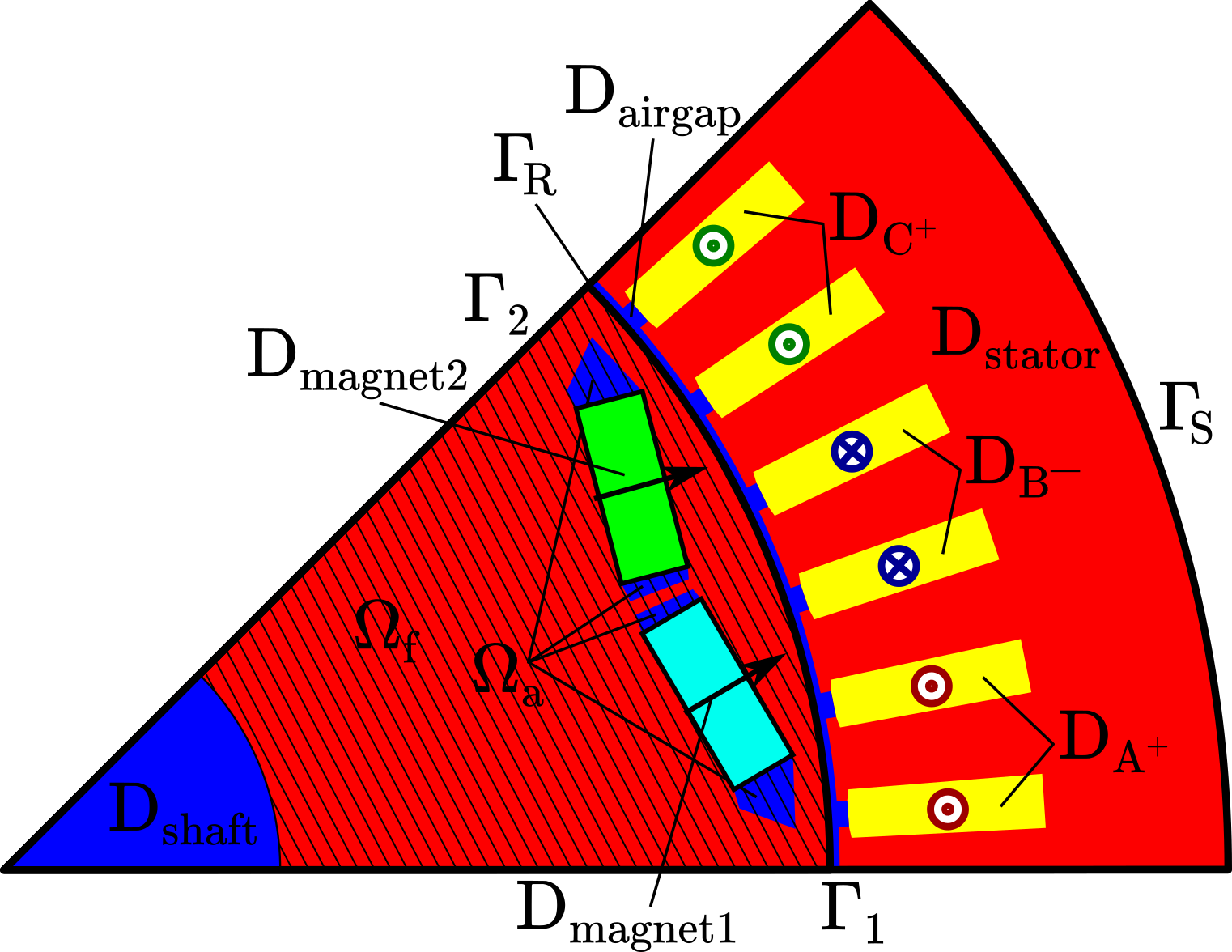}
    \caption{Left: 2d cross section of permanent magnet synchronous machine: Rotor and stator iron in red, coils in yellow with phase distribution (A: red, B: blue, C: green), permanent magnets with magnetization direction in light blue and light green, airgap, airpockets and shaft in blue. Right: One pole of the machine, the computational domain $D_\mathrm{all}$. Design domain $D$ (dashed area) consisting of rotor iron $\Omega_f\subseteq D$ (red) and rotor air $\Omega_a\subseteq D$ (blue). Rotor $D_R$ consisting of shaft $D_\mathrm{shaft},$ design domain $D$ and permanent magnets $D_\mathrm{magnet1}, D_\mathrm{magnet2}$ with outer boundary $\Gamma_R.$ Stator $D_S$ consisting of airgap $D_\mathrm{airgap}$, iron $D_\mathrm{stator}$ and coils $D_{A^+},D_{B^-},D_{C^+}$ with outer boundary $\Gamma_S$. Boundaries of pole $\Gamma_1,\Gamma_2$.}
    \label{fig:geometry}
\end{figure}
We consider a permanent magnet synchronous machine in static operation, see Figure \ref{fig:geometry} (left). The alternating source current which is impressed in the stator coils $D_A,D_B,D_C$, depicted in yellow, creates a rotating magnetic field. The interaction of this field with the field caused by the permanent magnets $D_\mathrm{magnet1}$ (light blue), $D_\mathrm{magnet2}$ (light green) produces a torque transferred by the shaft of the machine. Further we consider in our material model saturation effects  leading to a nonlinear equation. Since the iron parts $D_\mathrm{stator},\Omega_f$ (red) consist of thin isolated sheets stacked in axial direction, considering only a 2d cross section is a common approximation in electric machine optimization. Due to the periodicity of the design, the alternating directions of the permanent magnets and the antiperiodicity of the excitation currents the resulting fields are antiperiodic by one pole. It is therefore sufficient to consider a $45^\circ$ piece of the machine, see Figure \ref{fig:geometry} (right) as the computational domain imposing antisymmetric boundary conditions. These simplifications lead to the commonly used nonlinear magnetostatic approximation of the Maxwell's equations, i.e., to find the third component of the magnetic vector potential $u\in H^1(D_\mathrm{all})$ with 
\begin{align}\label{eq:initial_magnetostatic}\begin{split}
    \widetilde{\cur}h_\Omega(\cur u,q)&=j(q)\;\;\;\text{ in }D_\mathrm{all}\\
    u&=0\qquad\text{ on }\Gamma_S\\
    u|_{\Gamma_1}&=-u|_{\Gamma_2}
\end{split}
\end{align}
where $j(q)=j_{A}(q)\chi_{D_{A^+}}-j_{B}(q)\chi_{D_{B^-}}+j_{C}(q)\chi_{D_{C^+}}\in L_2(D_\mathrm{all})$ is the three phase source current and $q$ is a parameter which may be uncertain. By $\Omega:=(\Omega_f,\Omega_a)$ we denote the distribution of iron and air in the design domain $D$. The magnetic material law $h_\Omega$ defining the relation between magnetic flux $b=\cur u$ and magnetic field $h=h_\Omega(b,q)$ is given piecewise
\begin{align}
    \begin{split}h_\Omega(b,q)=&h(\Omega,b,q)\chi_{D}+h_f(b,q)\chi_{D_\mathrm{stator}}+h_a(b,q)\chi_{D_{A^+}\cup D_{B^-}\cup D_{C^+}\cup D_\mathrm{shaft}\cup D_\mathrm{airgap}}\\&+h_{m_1}(b,q)\chi_{D_\mathrm{magnet 1}}+h_{m_2}(b,q)\chi_{D_\mathrm{magnet 2}}
    \end{split}
    \label{eq:hpiecewise}
\end{align}
with $h(\Omega,b,q)=h_f(b,q)\chi_{\Omega_f}+h_a(b,q)\chi_{\Omega_a}$.
The choice of the material laws $h_f,h_a,h_m$ and the influence of the uncertain parameter $q$ will be specified in the section of numerical results, Section~\ref{sec: Numerical Results}. In order to be able to solve the magnetostatic equation \eqref{eq:initial_magnetostatic} for a rotor rotated by an arbitrary angle $\alpha$ we use the harmonic mortar approach from \citep{Egger_2022ab}. We allow $u_\alpha$ to be discontinuous at the interface $\Gamma_R$ and introduce a Lagrange multiplier $\lambda_\alpha$ to reinforce continuity in the potentially rotated setting:  Find $(u_\alpha,\lambda_\alpha)\in \mH\times\mL$ with
\begin{align}
\begin{split}
\int_{D_\mathrm{all}}h_\Omega(\cur u_\alpha,q)\cdot\cur v\dx + \langle\lambda_\alpha,(v|_{D_S}-v|_{D_R}\circ\rho_{-\alpha})\rangle_{\Gamma_R}&=\int_{D_\mathrm{all}}j(\alpha,q)v\dx\\
\langle\mu,(u_\alpha|_{D_S}-u|_{D_R}\circ\rho_{-\alpha})\rangle_{\Gamma_R} &=0
    \end{split}
    \label{eq:magnetostatic}
\end{align}
for all $(v,\mu)\in\mH\times\mL$ with $\mH=\{v:v|_{D_R}\in H^1(D_R), v|_{D_S}\in H^1(D_S), v|_{\Gamma_S}=0, v|_{\Gamma_1}=-v|_{\Gamma_2}\},\mL=H^{-\frac{1}{2}}(\Gamma_R).$ Based on energy considerations the torque produced by magneto-mechanical energy conversion can be computed by
\begin{align}\label{eq:torque}
    \T(\alpha)=T(u_\alpha,\lambda_\alpha,\alpha)=r_{\Gamma_R}\langle\lambda_\alpha,(\cur u_\alpha\cdot n_{\Gamma_R})\circ\rho_{-\alpha}\rangle_{\Gamma_R}
\end{align}
with $r_{\Gamma_R}$ the radius and $n_{\Gamma_R}$ the outer unit normal vector of the rotor. In the sequel we are interested in the average torque. Since the state $(u_\alpha,\lambda_\alpha)$ is periodic by mechanical rotations of $15^\circ$ we consider $N$ equidistant rotor positions
\begin{align}
    \alpha^n=\frac{15^\circ n}{N}, n=0,...,N-1
    \label{eq:rotorpositions}
\end{align}
within this period. We denote by $(u^n,\lambda^n):=(u_{\alpha^n},\lambda_{\alpha^n})$ the solution of \eqref{eq:magnetostatic} for the rotor position $\alpha^n$.
\begin{remark}\label{rem:diffusion}
    The differential operators $\widetilde{\cur},\cur$ introduced in Section~\ref{sec:notation} correspond to rotated versions of $\mathrm{div},\nabla$, respectively. For an isotropic material law, i.e. $h_\Omega(b_1,q)=h_\Omega(b_2,q)$ for all $b_1,b_2\in\R^2$ with $|b_1|=|b_2|$, problem \eqref{eq:initial_magnetostatic} corresponds to a nonlinear diffusion equation
    \begin{align*}
        \widetilde{\cur}h_\Omega(\cur u,q)=-\mathrm{div}h_\Omega(\nabla u,q),
    \end{align*}
    with suitable boundary conditions.
\end{remark}
\begin{remark}
    The formula for the torque \eqref{eq:torque} is derived in \citep{Egger_2022ab}. There are strong similarities to the widely used torque calculation method based on the Maxwell stress tensor \citep{Sadowski1992},
    \begin{align*}
        T=\int_S\frac{1}{\mu_0}B_rB_tr\ \mathrm{d}S,
    \end{align*} where $\mu_0$ is the magnetic vacuum permeability, $B_r$ and $B_t$ are the radial and normal component of the magnetic flux density $B$ and $S$ is a circle in the airgap. In our formula the Lagrange multiplier $\lambda$ represents the tangential component of the magnetic field $h\cdot t_{\Gamma_R}\sim \mu_0^{-1}B_t$ and $\cur u\cdot n_{\Gamma_R}=b\cdot n_{\Gamma_R}\sim B_r$. Similarly the domain of integration $\Gamma_R\subset D_\mathrm{airgap}$ is a circle in the airgap.
\end{remark}
\begin{remark}
    In \citep{Egger_2022ab} the authors propose to use harmonic basis functions to discretize the space of Lagrange multipliers $\mL$ which is also done in this work. However, to guarantee unique solvability one needs to fulfill a discrete inf-sup condition which results in a limitation of the degrees of freedom of the discrete space.
\end{remark}

\section{Methodology}
\label{sec: Methodology}
Although our theoretical considerations hold generally, we choose  a model problem to develop our methodology. 
\subsection{Model Optimization Problem}
Based on the forward problem introduced in Section~\ref{sec:Physical Model} we aim to find a design $\Omega$ which maximizes the average torque of the machine for some parameter $q$. This leads to the following PDE constrained optimization problem:
\begin{align}\label{eq:PROBLEM}
    \begin{aligned}
        &\min_\Omega-\frac{1}{N}\sum_{n=0}^{N-1}T(u^n,\lambda^n,\alpha^n) \\
        \text{ s.t. }e(\Omega,q,&\,\alpha^n,(u^n,\lambda^n))=0,\ n=0,...,N-1,
    \end{aligned}
\end{align}
where $T$ is the torque \eqref{eq:torque} and $e$ is the nonlinear magnetostatic PDE \eqref{eq:magnetostatic} for material configuration $\Omega$, parameter $q$ and rotor position $\alpha^n$. For a given design $\Omega$ and parameter $q$ we denote by $(u^n(\Omega,q),\lambda^n(\Omega,q))\in \mH\times\mL$ the unique solution of $e(\Omega,q,\alpha^n,(u,\lambda))=0$. The reduced cost functional of the parametrized design optimization problem is defined as 
\begin{align}\label{eq:reducedPROBLEM}
    \J(\Omega,q):=-\frac{1}{N}\sum_{n=0}^{N-1}T(u^n(\Omega,q),\lambda^n(\Omega,q),\alpha^n).
\end{align}
The Lagrangian of problem \eqref{eq:PROBLEM} is given by
\begin{align}
\label{eq:Lagrangian}
\begin{split}
    \mathcal{G}(\Omega,&q,(\alpha^0,...,\alpha^{N-1}),(u^0,...,u^{N-1},\lambda^0,...,\lambda^{N-1}),(p^0,...,p^{N-1},\eta^0,...,\eta^{N-1}))=\\&-\frac{1}{N}\sum_{n=0}^{N-1}T(u^n,\lambda^n,\alpha^n)+\sum_{n=0}^{N-1}\langle e(\Omega,q,\alpha^n,(u^n,\lambda^n)),(p^n,\eta^n)\rangle.
\end{split}
\end{align}
Since we will apply gradient based methodologies to solve the optimization problem \eqref{eq:PROBLEM}, we introduce the adjoint states $(p^n,\eta^n)\in\mH\times\mL$, which solve the adjoint equation, obtained by differentiation of the Lagrangian \eqref{eq:Lagrangian} with respect to the state $(u^n,\lambda^n)$,
\begin{align}
\begin{split}
\int_{D_\mathrm{all}}\partial_bh_\Omega(\cur u^n,q)\cur p^n\cdot\cur v^n\dx +& \langle\eta^n,(v^n|_{D_S}-v^n|_{D_R}\circ\rho_{-\alpha^n})\rangle_{\Gamma_R}\\&=\frac{r_{\Gamma_R}}{N}\langle\lambda^n,(\cur v^n\cdot n_{\Gamma_R})\circ\rho_{-\alpha^n}\rangle_{\Gamma_R}\\
\langle\mu^n,(p^n|_{D_S}-p^n|_{D_R}\circ\rho_{-\alpha^n})\rangle_{\Gamma_R} &=\frac{r_{\Gamma_R}}{N}\langle\mu,(\cur u^n\cdot n_{\Gamma_R})\circ\rho_{-\alpha^n}\rangle_{\Gamma_R}
    \end{split}
    \label{eq:adjoint}
\end{align}
for all $(v^n,\mu^n)\in\mH\times\mL, n=0,...,N-1$.
\subsection{Topology Optimization}
\label{sec:Topology Optimization}
 In this subsection we consider a fixed parameter $q$ and, for the sake of readability, skip it by writing $\J(\Omega)$ instead of $\J(\Omega,q).$ We denote the configuration of iron and air in the rotor by $\Omega=(\Omega_f,\Omega_a)\in\E,$ where $\Omega_f,\Omega_a$ are open and disjoint subsets of the design domain $D$ with $\overline{\Omega_f}\cup\overline{\Omega_a}=\overline{D}$. Here, $\E = E\times E$ denotes the set of admissible configurations where $E$ is the set of Lipschitz subsets of $D$ with a uniform Lipschitz constant $L_E$, which allows to prove existence of a minimizer, see also \citep{HenrotPierre2005, Gangl_2015aa}. We will use the notation $z\in\Omega :\Leftrightarrow z\in\Omega_f\cup\Omega_a$ to indicate that a point $z$ is in $D$ but not at the interface between $\Omega_f$ and $\Omega_a$. The sensitivity of $\J$ to circular topological perturbations $\omega_\epsilon(z):=\epsilon\omega+z$, see Figure \ref{fig:TDdef}, is given by the topological derivative. For perturbations of iron by including air the topological derivative is defined by
 \begin{figure}
     \centering
     \includegraphics[width=0.5\textwidth]{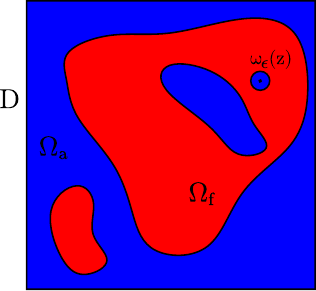}
     \caption{Some design domain $D$ consisting of two materials air in  $\Omega_a$ and iron in $\Omega_f$. Perturbation of $\Omega_f$ by putting air in $\omega_\epsilon(z)$.}
     \label{fig:TDdef}
 \end{figure}
\begin{align}
\label{eq:TDdef}
\frac{\id^{f\rightarrow a}}{\id\Omega}\J(\Omega)(z):=\lim_{\epsilon\searrow0}\frac{\J(\Omega_\epsilon(z))-\J(\Omega)}{|\omega_\epsilon(z)|}, \Omega_\epsilon(z)=(\Omega_f\setminus\overline{\omega_\epsilon(z)},\Omega_a\cup\overline{\omega_\epsilon(z)})
\end{align}
for $z \in \Omega_f$,
and for changes from air to iron in points $z\in\Omega_a$ by interchanging the role of iron (index $f$) and air (index $a$).  Combining these two, we get the topological derivative for $z\in\Omega$,
\begin{align}
    \frac{\id}{\id\Omega}\J(\Omega)(z)=\begin{cases}
        \frac{\id^{f\rightarrow a}}{\id\Omega}\J(\Omega)(z) &\text{ if }z\in\Omega_f,\\
        \frac{\id^{a\rightarrow f}}{\id\Omega}\J(\Omega)(z) &\text{ if }z\in\Omega_a.
    \end{cases}
\label{eq:GenTopologyDerv}
\end{align}
Based on this we will use the following term of optimality \citep{Amstutz_2006ab}: 
\begin{definition}
\label{def:Optimality}
    A configuration $\Omega\in\E$ is called locally optimal with respect to topological perturbations if 
    \begin{align}
        \frac{\id}{\id\Omega}\J(\Omega)(z)\ge 0 &\text{ for }z\in\Omega.
    \end{align}
\end{definition}
\subsubsection{Evaluation of the Topological Derivative}
The topological derivative for problems constrained by the magnetostatic PDE was first derived in \citep{Amstutz_2019aa}. To evaluate its formula in a systematic way we apply the framework of \citep{GANGLAutoTopo}. For a point $z\in\Omega$ we use the abbreviation $U^n_z=\cur u^n(z)$ and $P^n_z=\cur p^n(z)$ for point evaluation of the state $u^n$ coming from \eqref{eq:magnetostatic} and the adjoint state $p^n$ coming from \eqref{eq:adjoint}, respectively. First we introduce the first order asymptotic corrector of the state subject to topological perturbations $K_U\in BL(\R^2):=\{v\in L_2^{loc}(\R^2):\nabla v\in L_2(\R^2)^2\}/\R$ which is the solution of the the auxiliary exterior problem for $U\in\R^2$
\begin{align}
    \label{eq:exterior}
    \begin{aligned}
    \int_{\R^2}(h_\omega^{f\rightarrow a}(\xi, \cur K_U(\xi)+U)-&h_\omega^{f\rightarrow a}(\xi, U))\cdot\cur v(\xi)\ \id\xi \\
    &=-\int_\omega(h_a(U)-h_f(U))\cdot\cur v(\xi)\ \id\xi
    \end{aligned}
\end{align}
for all $v\in BL(\R^2)$ with  $\omega=B_1(0)$ and $ h_\omega^{f\rightarrow a}(\xi, b)=h_a(b)\chi_{\omega}(\xi)+ h_
f(b)\chi_{\R^2\setminus\overline{\omega}}(\xi)$. The topological derivative for material changes from iron to air, i.e. $z\in\Omega_f$, is then given by
\begin{align}
\label{eq:TDformula}
\begin{split}
    \frac{\id^{f\rightarrow a}}{\id\Omega}&\J(\Omega)(z)=\\&\sum_{n=0}^{N-1}\bigg(\frac{1}{|\omega|}\int_{\R^2}(h_\omega^{f\rightarrow a}(\xi, \cur K_{U_z^n}(\xi)+U_z^n)-h_\omega^{f\rightarrow a}(\xi, U_z^n)-\partial_bh_\omega^{f\rightarrow a}(\xi, U_z^n)\cur K_{U_z^n}(\xi))\cdot P_z^n\ \id\xi\\
    &+\frac{1}{|\omega|}\int_\omega(\partial_b h_a(U_z^n)-\partial_b h_f(U_z^n))\cur K_{U_z^n}(\xi)\cdot P_z^n\ \id\xi\\
    &+(h_a(U_z^n)-h_f(U_z^n))\cdot P_z^n\bigg),
\end{split}
\end{align}
where only the piecewise defined material law $h_\omega^{f\rightarrow a}$ and $K_{U_z^n}$ depend on $\xi$ and $U^n_z,P^n_z\in\R^2$ are constant. If the material laws are both linear of the form
\begin{align}
    h_f(b)=\nu_fb, h_a(b)=\nu_0b,
\end{align}
the exterior problem \eqref{eq:exterior} has an analytical solution
\begin{align}
    -\frac{\nu_0-\nu_f}{\nu_0+\nu_f}U\cdot\left(\xi\chi_\omega(\xi)+\frac{\xi}{|\xi|^2}\chi_{\R^2\setminus\overline{\omega}}(\xi)\right)
\end{align}
which leads to the topological derivative formula
\begin{align}
    \frac{\id^{f\rightarrow a}}{\id\Omega}&\J(\Omega)(z)=\sum_{n=0}^{N-1}2\nu_0\frac{\nu_0-\nu_f}{\nu_0+\nu_f}U_z^n\cdot P_z^n.
\end{align}
\begin{remark}
    For linear material laws the evaluation of the topological derivative is straightforward. In the nonlinear case one theoretically has to solve the exterior problem \eqref{eq:exterior} for every point $z\in\Omega$, rotor position $n$ depending on the current design $\Omega.$ However, since the topological derivative depends only on $U_z^n=\cur u^n(z)\in\R^2$ and $P_z^n=\cur p^n(z)\in\R^2, z\in\Omega, n=0,...,N-1$ it is convenient to take samples for all possible $U\in\R^2, P\in\R^2$ and to interpolate them. First we note, that formula \eqref{eq:TDevaluate} depends on $P_z^n$ in a linear way, therefore it is sufficient to consider only the base vectors of $\R^2$, $e_0,e_{\pi_2}$, for $P.$ Further it was shown in \citep{Amstutz_2019aa} that due to the isotropic characteristic of the material laws $h_f,h_a$ the topological derivative can be written as
    \begin{align}
        \frac{\id^{f\rightarrow a}}{\id\Omega}\J(\Omega)(z)=\sum_{n=0}^{N-1}\begin{pmatrix}
            \cos\beta_z^n&-\sin\beta_z^n\\\sin\beta_z^n&\cos\beta_z^n
        \end{pmatrix}\begin{pmatrix}
            f_1^{f\rightarrow a}(|U_z^n|)\\f_2^{f\rightarrow a}(|U_z^n|)
        \end{pmatrix}\cdot P_z^n, \beta_z^n=\arccos\frac{U_z^n\cdot e_0}{|U_z^n|}
        \label{eq:TDevaluate}
    \end{align}
    with
    \begin{align}\label{eq:TDsample}
    \begin{split}
        f_1^{f\rightarrow a}(t)=&\frac{1}{|\omega|}\int_{\R^2}(h_\omega^{f\rightarrow a}(\xi, \cur K_{te_0}(\xi)+te_0)-h_\omega^{f\rightarrow a}(\xi, te_0)-\partial_bh_\omega^{f\rightarrow a}(\xi, te_0)\cur K_{te_0}(\xi))\cdot e_0 \id\xi\\
    &+\frac{1}{|\omega|}\int_\omega(\partial_bh_a(te_0)-\partial_b h_f(te_0))\cur K_{te_0}(\xi)\cdot e_0\ \id\xi\\
    &+(h_a(te_0)-h_f(te_0))\cdot e_0,\\
    f_2^{f\rightarrow a}(t)=&\frac{1}{|\omega|}\int_{\R^2}(h_\omega^{f\rightarrow a}(\xi, \cur K_{te_0}+te_0)-h_\omega^{f\rightarrow a}(\xi, te_0)-\partial_bh_\omega^{f\rightarrow a}(\xi, te_0)\cur K_{te_0}(\xi))\cdot e_{\frac{\pi}{2}} \id\xi\\
    &+\frac{1}{|\omega|}\int_\omega(\partial_b h_a(te_0)-\partial_b h_f(te_0))\cur K_{te_0}(\xi)\cdot e_{\frac{\pi}{2}}\ \id\xi\\
    &+(h_a(te_0)-h_f(te_0))\cdot e_{\frac{\pi}{2}}.
    \end{split}
    \end{align}
It is therefore convenient to take samples $f_{1,k}^{f\rightarrow a}=f_{1}^{f\rightarrow a}(t_k), f_{2,k}^{f\rightarrow a}=f_2^{f\rightarrow a}(t_k)$ for $0=t_1<t_2<...<t_K=t_\mathrm{max}$ of formula \eqref{eq:TDsample} (which includes solving the exterior problem \eqref{eq:exterior} for $U=t_ke_0$) in an offline phase and providing an interpolation
\begin{align}
\tilde{f}_1^{f\rightarrow a}, \tilde{f}_2^{f\rightarrow a}.
    \label{eq:TDinterpolation}
\end{align}
which gets evaluated online when computing the topological derivative using~\eqref{eq:TDevaluate}.
\end{remark}
\begin{remark}
    The topological derivative from air to iron $\frac{\id^{a\rightarrow f}}{\id\Omega}$ can be computed by applying the same formula \eqref{eq:TDformula} but interchanging the role of the materials, i.e. interchanging $h_f$ and $h_a$. Note that this also requires the solution of the modified version of \eqref{eq:exterior}.
\end{remark}
\begin{remark}
    Note that, although we skipped the uncertainty $q$ in this subsection, all material laws may depend on the parameter by $h(b)=h(b,q)$, therefore also the topological derivative depends on $q$ including the precomputed functions $f_1^{f\rightarrow a}(t,q),f_2^{f\rightarrow a}(t,q),f_1^{a\rightarrow f}(t,q),f_2^{a\rightarrow f}(t,q)$.
\end{remark}

\subsubsection{Level Set Algorithm}
\label{sec:LevelSetAlg}
This methodology on how to update the design defined by a level set function using the topological derivative was introduced in \citep{Amstutz_2006ab}.
We represent the configuration $\Omega$ by a continuous level set function $\psi\in\mathcal{S}$
\begin{align}\label{eq:levelset}\begin{split}
    \psi(z)>0&\Leftrightarrow z\in\Omega_f\\
    \psi(z)=0&\Leftrightarrow z\in\overline{\Omega_f}\cap\overline{\Omega_a}\\
    \psi(z)<0&\Leftrightarrow z\in\Omega_a
\end{split}
\end{align}
with $\mathcal{S}\subset \left\{\psi\in C(D):\|\psi\|_{L_2(D)}=1\right\}$ the space of admissible level set functions. Next we introduce the generalized topological derivative
\begin{align}
    g_\Omega(z)=\begin{cases}
        \phantom{-}\frac{\id^{f\rightarrow a}}{\id\Omega}\J(\Omega)(z)&\text{ if }z\in\Omega_f\\
        -\frac{\id^{a\rightarrow f}}{\id\Omega}\J(\Omega)(z)&\text{ if }z\in\Omega_a.
    \end{cases}
    \label{eq:TDoriented}
\end{align}
This is motivated by the following lemma
\begin{Lemma}
\label{lemma}
    The configuration $(\Omega_f,\Omega_a)=\Omega_\psi$ described by a level set function $\psi\in\mathcal{S}$ is (locally) optimal with respect to topological perturbations if
    \begin{align}
    \label{eq:lemmacondition}
        \psi(z)=\frac{g_{\Omega_\psi}(z)}{\|g_{\Omega_\psi}\|_{L_2(D)}}
    \end{align}
    for all $z\in \Omega_\psi$.
\end{Lemma}
\begin{proof}
    Let $z\in\Omega_f$. Then we conclude
    \begin{align*}
        z\in\Omega_f\Leftrightarrow0<\psi(z)=\frac{g_{\Omega_\psi}(z)}{\|g_{\Omega_\psi}\|_{L_2(D)}}\Leftrightarrow g_{\Omega_\psi}(z)>0\Leftrightarrow \frac{\id^{f\rightarrow a}}{\id\Omega}\J(\Omega)(z)>0.
    \end{align*}
    Similarly we get for $z\in\Omega_a$ 
        \begin{align*}
        z\in\Omega_a\Leftrightarrow0>\psi(z)=\frac{g_{\Omega_\psi}(z)}{\|g_{\Omega_\psi}\|_{L_2(D)}}\Leftrightarrow g_{\Omega_\psi}(z)<0\Leftrightarrow \frac{\id^{a\rightarrow f}}{\id\Omega}\J(\Omega)(z)>0.
    \end{align*}
    which yields the pointwise optimality condition of Definition \ref{def:Optimality}.
\end{proof}
Based on the optimality condition \eqref{eq:lemmacondition},  \citep{Amstutz_2006ab} suggests to use a fixed point iteration to update the level set function. In Algorithm~\ref{alg:Levelset} we present this algorithm adapted to the nonlinear magnetostatic problem including an offline phase for the precomputation of the topological derivative~\eqref{eq:TDsample}.
\begin{algorithm}
\caption{Level Set Algorithm for Topology Optimization by Topological Derivative.}
    \label{alg:Levelset}
    \underline{Offline phase}
    \begin{algorithmic}
        \State \textbf{Choose} samples $0=t_1<t_2<...<t_K=t_\mathrm{max}$
        \For {$k=1,...,K$}
        \State \textbf{Solve} exterior problem \eqref{eq:exterior} for $U=t_ke_0$
        \State \textbf{Evaluate} $f_{1,k}^{f\rightarrow a},f_{2,k}^{f\rightarrow a},f_{1,k}^{a\rightarrow f},f_{2,k}^{a\rightarrow f}$ \eqref{eq:TDsample}
        \EndFor
    \State \textbf{Compute} interpolation $\tilde{f}_{1}^{f\rightarrow a},\tilde{f}_{2}^{f\rightarrow a},\tilde{f}_{1}^{a\rightarrow f},\tilde{f}_2^{a\rightarrow f}$ \eqref{eq:TDinterpolation}
    \end{algorithmic}
    \underline{Online phase}
\begin{algorithmic}
    \State\textbf{Choose} $\psi_0\in\mathcal{S}, k=0,k_\mathrm{max}<\infty, \varepsilon>0, 0<s_\mathrm{min}<s_\mathrm{max}\le 1,\gamma\in(0,1), \delta\ge1$
    \State \textbf{Evaluate} $\J(\Omega_{\psi_0})$ (evaluate \eqref{eq:reducedPROBLEM} by solving state equation \eqref{eq:magnetostatic})
    \For{$k=0,...,k_\mathrm{max}$}
    \State \textbf{Compute} $g_{k}=g_{\Omega_{\psi_k}}$ (solve adjoint equation \eqref{eq:adjoint}, evaluate \eqref{eq:TDinterpolation},\eqref{eq:TDevaluate}, apply \eqref{eq:TDoriented})
    \If{$\theta_k=\arccos\frac{\left(\psi_k,g_{k}\right)_{L_2(D)}}{\|g_{k}\|_{L_2(D)}}<\varepsilon$}
    \State\textbf{break}
    \EndIf
    \State $s=s_\mathrm{max}$
    \While{$s> s_\mathrm{min}$}
    \State $\psi_{k+1}=\frac{1}{\sin\theta_k}\left(\sin((1-s)\theta_k)\psi_k+\sin(s\theta_k)\frac{g_{k}}{\|g_{k}\|_{L_2(D)}}\right)$
    \If {$\J(\Omega_{\psi_{k+1}})<\J(\Omega_{\psi_k})$} (evaluate \eqref{eq:reducedPROBLEM} by solving \eqref{eq:magnetostatic})
    \State{$s=\min\{s_\mathrm{max},\delta s\}$}
    \State\textbf{break}
    \Else
    \State $s=\max\{s_\mathrm{min},\gamma s\}$
    \EndIf
    \EndWhile
    \EndFor
\end{algorithmic}
\end{algorithm}
The condition \eqref{eq:lemmacondition} of Lemma \ref{lemma} is fulfilled if and only if the algorithm has reached a stationary point, i.e. $\psi^{k+1}=\psi^{k}$. By the line search we guarantee, that this stationary point is indeed a local minimizer of the optimization problem.
\begin{remark}
    Since the generalized topological derivative $g$ is not necessarily continuous, we need an additional step to ensure $\psi_k\in\mathcal{S}, k\ge 0.$ We do this by replacing $g$ by the solution $\tilde{g}$ of the PDE
    \begin{align*}
        -\varepsilon\Delta \tilde{g}+\tilde{g}&=g\text{ in }D\\
        \nabla \tilde{g}\cdot n&=0\text{ on }\partial D.
    \end{align*}
    If $D$ and $g$ are sufficiently regular we get $\tilde{g}\in H^{1+\delta}(D)\hookrightarrow C(D)$ since $D\subset\R^2$. This procedure is similar to sensitivity filtering in density based optimization, the smoothing parameter $\varepsilon$ is related to the minimal feature size, see e.g. \citep{Lazarov2011}. The condition $\|\psi_k\|_{L_2(D)}=1$ is fulfilled by construction of the update step as spherical linear interpolation.
\end{remark}
\subsection{Robust Optimization}
\label{sec:robust optimization}
Algorithm~\ref{alg:Levelset} is a powerful tool for finding locally optimal topologies under deterministic conditions. However, practical engineering designs - especially in the context of electrical machines - are rarely free from uncertainty. To ensure that solutions remain effective and reliable under such real-world conditions, it is important to incorporate uncertainty into the optimization process. To address these challenges, we now present a general framework for solving a robust optimization problem. As discussed in the introduction in Section~\ref{Introduction}, uncertainties can arise from both design variables \( x \) and model parameters \( q \). In practical applications, uncertainties in model parameters $q$ may arise from measurement errors, model approximations, or environmental conditions and were already introduced as parameter uncertainty. Similarly, uncertainties in design variables \( x \) can result from manufacturing tolerances. In optimization without considering uncertainties, the problem is often solved to a high numerical accuracy. However, the precise optimized design cannot be achieved in practice due to limitations in manufacturing and implementation. To handle these uncertainties, we define an uncertainty set $U=U_x\times U_q$ where $U_x$ contains the possible variations of the design variable $\delta x$ and $U_q$ all possible uncertain values of the parameters $q$ that are in a region around the nominal value $\hat{q} \in U_q$.
These sets can be defined using prior knowledge, statistical data, or uncertainty quantification and are fixed since we are using a deterministic approach. Choosing a reasonable region around the nominal value ensures that the robust solution performs well without being overly conservative.
\begin{remark}
A common choice for the uncertainty sets are closed ellipsoids:
    \begin{itemize}
    \item Uncertainty set for model parameters:
    \begin{equation*}
        U_{q} = \left\{ q \in \mathbb{R}^{n_q} \mid q = \hat{q} + R_q v_q: v_q\in\R^{n_q}, \ \|v_q\|_2 \leq 1 \right\},
    \end{equation*}
    where \( \hat{q} \in \mathbb{R}^{n_q} \) is the nominal parameter, and \( R_q \in \mathbb{R}^{n_q \times m_q} \) is a  matrix with \(\operatorname{rank}(R_q) = m_q \leq n_q\) that defines the shape and size of the ellipsoidal uncertainty region in the parameter space.
\vspace{0.5cm}
    \item Uncertainty set for design variables:
    \begin{equation*}
        U_{x} = \left\{ \delta x \in \mathbb{R}^{n_x} \mid \delta x = R_x v_x: v_x\in\R^{n_x}, \ \|v_x\|_2 \leq 1 \right\},
    \end{equation*}
    where \( R_x \in \mathbb{R}^{n_x \times m_x} \) is also a matrix with \(\operatorname{rank}(R_x) = m_x \leq n_x\) and characterizes the admissible design perturbations.
\end{itemize}
\end{remark}
We now derive the robust optimization formulation similar to \citep{Komann2024,Lass_2017aa} of the nominal problem
\begin{equation} \min_{x} g(x, q) \label{eq:NominalProblem} \end{equation}
with $g:\R^{n_x}\times\R^{n_q}\rightarrow\R$ a real valued function. Incorporating the uncertainties in the design variables $\delta x\in U_x$ and model parameters $q\in U_q$ yields the robust optimization problem
\begin{equation}
    \min_{x} \max_{(\delta x, q) \in U_x \times U_q} g(x + \delta x, q).
    \label{eq:RobustFormulation}
\end{equation}
Next we introduce the worst case function which is the value function of the inner maximization problem
\begin{equation}
    f(x) = \max_{(\delta x, q) \in U_x \times U_q} g(x + \delta x,  q).
    \label{eq:WorstCaseFunction}
\end{equation}
Using this, we can rewrite the robust optimization problem \eqref{eq:RobustFormulation} as
\begin{equation}
    \min_{x} f(x)
    \label{eq:OptimizationProblemRobust}
\end{equation}
\begin{remark}
    In our application $g$ is the reduced cost functional of a PDE constrained optimization problem as introduced in \eqref{eq:reducedPROBLEM}.
\end{remark}
\begin{remark}
Solving \eqref{eq:RobustFormulation} directly by brute force, i.e. sampling the uncertainty set, is computationally expensive and often intractable. A comprehensive overview of problems of the form \eqref{eq:RobustFormulation} can be found in \citep{Leyffer2020}. Some robust optimization problems can be related to bilevel optimization problems. The connections between robust optimization and bilevel optimization are discussed in detail in \citep{Schmidt2024}.  To obtain tractability for the inner maximization \((\delta x,q)\mapsto g(x+\delta x,q)\) in \eqref{eq:WorstCaseFunction}, $g$ is typically approximated using either a linear \citep{Diehl_2006aa} or quadratic Taylor expansion \citep{Lass_2017aa}, where  the approximated problem is then solved. However, while a linear approximation provides an analytical solution formula, the approximation quality suffers for highly nonlinear problems. A quadratic approximation offers higher accuracy and leads to a Trust Region problem, but requires up to third-order derivatives, which is very tedious in cases involving PDE constraints. 
\end{remark}
In this work, we directly address the inner maximization problem and apply a theorem  of Clarke on generalized gradients of value functions \citep[Thm. 2.1]{clarke1975generalized}, which under quite general assumptions requires only first-order derivatives \citep{Danskin1967,Bertsimas_2010aa}. A related approach is used in \citep{Bertsimas_2010aa}, where descent directions of the worst case functions are used. We rely on the following version of \citep[Thm. 2.1]{clarke1975generalized}, see also \citep[Thm. 1]{Komann2024}.

\begin{Theorem}
Let \( U_x \times U_q \subset \mathbb{R}^{n_x} \times \mathbb{R}^{n_q} \) be convex and compact, and let \( g: \mathbb{R}^{n_x} \times U_q \to \mathbb{R} \) be continuous and continuously differentiable with respect to \( x \). Then the following properties hold for \( f(x) \), as defined in \eqref{eq:WorstCaseFunction}:
\begin{enumerate}
    \item \( f(x) \) is locally Lipschitz continuous and directionally differentiable.
    \item If \( U^*(x) \) denotes the set of maximizers \( (\delta x^*, q^*) \) in \eqref{eq:WorstCaseFunction}, then Clarke's generalized gradient of \( f(x) \) is given by:
    \begin{equation}\label{eq:fderivative}
        \partial f(x) = \operatorname{conv}\left\{ \nabla_x g(x + \delta x^*, q^*) \mid 
        (\delta x^*, q^*) \in U^*(x) \right\}.
    \end{equation}
    \item If \( U^*(x) \) contains a single element, then \( f(x) \) is differentiable at \( x \), and \eqref{eq:fderivative} provides the classical gradient.
\end{enumerate}
\label{thm:classical}
\end{Theorem}
\begin{proof}
The proof can be found in \cite[Thm. 2.1]{clarke1975generalized}.
\end{proof}
In summary, this approach has the advantage that only first-order derivatives are required, and with \eqref{eq:fderivative} we have a direct formula for computing the subgradient with respect to the worst case function. Compared to solving the robust problem \eqref{eq:RobustFormulation} directly, the additional cost is given by finding the worst case of the inner maximization problem, as the outer optimization is similar to the nominal problem \eqref{eq:NominalProblem}, but usually nonsmooth. So a crucial step is to determine a maximizer \( (\delta x^*, q^*) \in U^*(x) \) to evaluate $f(x)$ and $g\in\partial f(x)$. 

\begin{remark} A related approach can be found in \citep{Caubet2024}, where Clarke's subdifferential is used to solve robust shape optimization problems of the form \eqref{eq:OptimizationProblemRobust} with uncertain parameters. In contrast, we focus in this work on topology optimization using the topological derivative. \end{remark}

\section{ Robust Topology Optimization }
\label{sec: Combined Approach}
In this section we present a combination of the level set based topology optimization using the topological derivative from Section~\ref{sec:Topology Optimization} and the robust optimization considering the worst case from Section~\ref{sec:robust optimization}. Since the set of designs $\E$ is not a normed vector space, uncertainties on the design cannot be considered in a straightforward way by the presented theory. Thus, we restrict ourselves to uncertainties in the additional parameter $q\in U\subset\R^{n_q}$ and do not consider uncertainties in the optimization variable $\Omega\in\E$ (corresponding to $x\in\R^{n_x}$ in \eqref{eq:NominalProblem}). The extension to geometric uncertainties is left open for future research. Our goal is to solve the robust topology optimization problem
\begin{align}\label{eq:robustPROBLEM}
    \min_{\Omega\in\E}\max_{q\in U}\J(\Omega,q),
\end{align}
with $\J$ as defined in \eqref{eq:reducedPROBLEM}.
\subsection{Theoretical result}
For a functional $\J:\E\times U\rightarrow\R$ we define the worst case function with respect to $q$ by
\begin{align}
    \varphi(\Omega):=\max_{q\in U}\J(\Omega,q).
    \label{eq:WCfuncdef}
\end{align}
  In order to show the main theoretical result we need the following assumption.
 \begin{assumption}\label{ass:ass}
     Let $\J:\E\rightarrow\R$ be a real-valued shape function. We assume that there exists a $\tau_0>0$
     such that the mapping
     \begin{align}\label{eq:assumption}
         \tau\mapsto\frac{1}{|\omega|}\J(\Omega_{\sqrt{\tau}}(z))
     \end{align}
     is continuously differentiable for all $\tau\in[0,\tau_0]$ and all $z\in\Omega$ with $\Omega_\epsilon(z)$ and ${\omega_\epsilon(z)=\epsilon\omega+z}$, ${\omega=B_1(0)}$ as in \eqref{eq:TDdef}.
 \end{assumption}
 We now show that this assumption implies the existence of the topological derivative of $\Omega\mapsto\J(\Omega)$
 \begin{Lemma}\label{lem:gtilde}
     Let $\J:\E\times U\rightarrow\R$ be a real valued function. Assume that Assumption \ref{ass:ass} holds for $\J(\cdot,q)$, for all $q\in U$. Then, for $z\in\Omega$, the function
     \begin{align}\label{eq:gtilde}
    \tau\mapsto \tilde{g}(\tau,q)=\tfrac{1}{|\omega|}\J(\Omega_{\sqrt{\tau}}(z)
,q)
\end{align}
 is continuously differentiable with respect to $\tau$ and the derivative in $0$ coincides with the topological derivative
\begin{align}
    \frac{\partial^+}{\partial\tau}\tilde{g}(0,q)=\frac{\partial}{\partial\varOmega} \J(\Omega,q):=\lim_{\epsilon\searrow 0}\frac{\J(\Omega_\epsilon(z),q)-\J(\Omega,q)}{|\omega_\epsilon(z)|}
\end{align}
for all $q\in U$.
 \end{Lemma}
 \begin{proof}
     Let $z\in\Omega,q\in U$ be arbitrary but fixed. The partial differentiability of $\tilde{g}$ follows by assumption. We conclude
     \begin{align}\label{eq:diffgtilde}
\frac{\partial^+}{\partial \tau}\tilde{g}(0,q)=\lim_{\delta\searrow 0}\frac{\tilde{g}(\delta,q)-\tilde{g}(0,q)}{\delta}\overset{\delta=\epsilon^2}{=}\lim_{\epsilon\searrow 0}\frac{\J(\Omega_\epsilon(z))-\J(\Omega)}{|\omega|\epsilon^2},
\end{align}
which gives the desired identification.
 \end{proof}
 We now prove that the topological derivative of the worst case objective function exists and can be written in a simple form using Theorem \ref{thm:classical} which is based on \citep[Thm 2.1]{clarke1975generalized}.
\begin{Theorem}
Let $U\subset\R^{n_q}$ be a convex compact set and
$\J:\E\times U\rightarrow\R$ be a real valued objective function. We assume that, for all $q\in U$, Assumption \ref{ass:ass} holds for $\J(\cdot,q)$. Furthermore we assume that the set of maximizers $U^*(\Omega)$ of the mapping $q\mapsto\J(\Omega,q)$ is a singleton and define the worst case parameter
\begin{align}
    q^*(\Omega):=\underset{q\in U}{\mathrm{arg\,max}}\, \J(\Omega,q),
    \label{eq:WCdef}
\end{align} for $\Omega\in\E$. Then the topological derivative of the worst case function \eqref{eq:WCfuncdef}, \( \frac{\mathrm{d}}{\mathrm{d}\varOmega}\varphi(\Omega)\),  exists for $z\in \Omega$ and is given by:
 \begin{equation}
 \frac{\mathrm{d}}{\mathrm{d}\varOmega}\varphi(\Omega)(z)=\frac{\partial}{\partial\varOmega} \J(\Omega,q^*(\Omega)).
        \label{eq:GeneralizedDerivative}
  \end{equation}
  \label{thm:robust}
\end{Theorem}
\begin{proof}
 By Lemma \ref{lem:gtilde}, the function $\tau\mapsto\tilde{g}(\tau,q)$, defined in \eqref{eq:gtilde}, is continuously differentiable on $[0,\tau_0]$ for $q\in U$. This yields the existence of a partially continuously differentiable extension ${g:\R\times U\rightarrow\R}$ with 
 \begin{align*}
     g(\tau,q)=\tilde{g}(\tau,q)=\frac{1}{|\omega|}\J(\Omega_{\sqrt{\tau}},q),
 \end{align*}
 for $\tau\in[0,\tau_0], q\in U$. For $g$ we define the worst case function $f:\R\rightarrow\R,$
\[
    f(\tau) := \max_{q\in U}g(\tau,q).
\]
Note, that for $\tau\in[0,\tau_0]$ the worst cases of $g$ and $\J$ coincide
\[q^*(\tau):=\underset{q\in U}{\mathrm{arg\,max}}\,g(\tau,q)=\underset{q\in U}{\mathrm{arg\,max}}\,\J(\Omega_{\sqrt{\tau}},q)=q^*(\Omega_{\sqrt{\tau}}).\] Similarly we can identify the worst case functions for $\tau\in[0,\tau_0]$
\begin{align*}
    f(\tau) = g(\tau,q^*(\tau))= \frac{1}{|\omega|}\J(\Omega_{\sqrt{\tau}},q^*(\Omega_{\sqrt{\tau}}))=\frac{1}{|\omega|}\varphi(\Omega_{\sqrt{\tau}}).
\end{align*}
We apply Theorem \ref{thm:classical} to the function $g:\R\times U\rightarrow\R$ and get that the derivative of \( f \) exists in $0$ and can be computed as follows:
\begin{align}\label{eq:worstcaseg}
    \frac{\mathrm{d}}{\mathrm{d} \tau}f(0)=\frac{\partial}{\partial \tau}g(0,q^*(0))=\frac{\partial^+}{\partial\tau}\tilde{g}(0,q^*(0)).
\end{align}
Plugging in the definition of $f$ we get
\begin{align*}
\frac{\id}{\id\Omega}\varphi(\Omega)(z)=\lim_{\epsilon\searrow 0}\frac{\varphi(\Omega_\epsilon(z))-\varphi(\Omega)}{|\omega_\epsilon(z)|}\overset{\tau=\epsilon^2}{=}\lim_{\tau\searrow0}\frac{f(\tau)-f(0)}{\tau}=\frac{\mathrm{d}}{\mathrm{d}\tau}f(0)
\end{align*}
which, together with \eqref{eq:worstcaseg} and \eqref{eq:diffgtilde}, gives the desired result.
\end{proof}
Theorem \ref{thm:robust} states that, under the given assumptions, the robust topological derivative, i.e. the topological derivative of the worst case function $\varphi$ is given by evaluating the topological derivative of the nominal problem at the worst case parameter $q^*(\Omega)$. 
This leads to a natural extension of the level set algorithm to robust topology optimization problems.
\begin{corollary}\label{cor}
    Let $\varphi(\Omega):=\max_{q\in U}\J(\Omega,q)$ bet the worst case function of the robust design optimization problem \eqref{eq:robustPROBLEM} with a linear PDE constraint, i.e. with a linear material law $h_f(b)=\nu_fb$ in \eqref{eq:magnetostatic}. Assume that the set of maximizers $U^*(\Omega):=\{q^*\in U:\J(\Omega,q^*)=\varphi(\Omega)\}$ is a singleton for all $\Omega\in \E$. Then $\varphi$ is topologically differentiable for $z\in\Omega$ with
    \begin{align*}
        \frac{\mathrm{d}}{\mathrm{d}\varOmega}\varphi(\Omega)(z)=\frac{\partial}{\partial\varOmega} \J(\Omega,q^*(\Omega)).
    \end{align*}
\end{corollary}
\begin{proof}
    Assumption \ref{ass:ass} for $\J(\cdot,q)$ is shown in Lemma \ref{lem:app} in the Appendix. Theorem \ref{thm:robust} yields the desired result
\end{proof}
\begin{remark}
    For the more general case of a nonlinear relation $h_f=f(b)$ which is usually used in electric machine simulation, a result corresponding to Lemma \ref{lem:app} is currently not available in the literature. This is subject of ongoing and future research.
\end{remark}
\begin{remark}
    Of course assuming the uniqueness of the worst case $q^*(\Omega)$ is a strong assumption. However Theorem \ref{thm:classical} states that if the set of maximizers $U^*(x)$ is not a singleton one should continue with an element of the convex hull of derivatives for all elements of $U^*(x)$. In particular, taking the gradient evaluated in one element of $U^*(x)$ is valid. We will follow this approach in the numerical realization by choosing the element we obtain by numerically maximizing $q\mapsto\J(\Omega,q)$ using Algorithm \ref{alg:InnerProblem} as $q^*(\Omega)$.
\end{remark}

\subsection{Numerical Algorithm}
We define the generalized robust topological derivative for $z\in\Omega$ by
\begin{align}
    g_{\Omega,q}(z)=\begin{cases}
        \phantom{-}\frac{\id^{f\rightarrow a}}{\id\Omega}\J(\Omega,q^*(\Omega))(z)&\text{ if }z\in\Omega_f\\
        -\frac{\id^{a\rightarrow f}}{\id\Omega}\J(\Omega,q^*(\Omega))(z)&\text{ if }z\in\Omega_a.
    \end{cases}
    \label{eq:TDorientedRobust}
\end{align}
Adding the computation of the worst case $q^*$ to the level set algorithm Algorithm~\ref{alg:Levelset}, we arrive at the robust counterpart sketched in Algorithm~\ref{alg:RobustLevelset}.
\begin{algorithm}
    \caption{Parameter Optimization Algorithm of the Inner Problem.}
    \label{alg:InnerProblem}
    \begin{algorithmic}
    \State\textbf{Choose} $q^0\in U, l_\mathrm{max}<\infty, \varepsilon_\tau>0, 0<\tau_\mathrm{min}<\tau_\mathrm{max}<\infty,\gamma_\tau\in(0,1),\delta_\tau\ge1, \gamma\in\bigl(0,\tfrac12\bigr)$
    \State\textbf{Evaluate} $\J(\Omega,q^0)$ (evaluate \eqref{eq:reducedPROBLEM} by solving state equation \eqref{eq:magnetostatic})
    \For{$l=0,...,l_\mathrm{max}$}
    \State \textbf{Compute} $\nabla_q\J(\Omega,q^l)$ (solve adjoint equation \eqref{eq:adjoint}, evaluate \eqref{eq:gradientq})
    \State $\tau=\tau_\mathrm{max}$
    \While{$\tau>\tau_\mathrm{min}$}
    \State{$q^{l+1}=\mathcal{P}_U(q^l+\tau\nabla_q\J(\Omega,q^l))$  }
    \If {$\J(\Omega,q^{l+1})-\J(\Omega,q^l) \ge \gamma/\tau \|q^{l+1}-q^l\|^2 $} (solve \eqref{eq:magnetostatic}, evaluate \eqref{eq:reducedPROBLEM})
    \State{$\tau=\min\{\tau_\mathrm{max},\delta_\tau\tau\}$}
    \State\textbf{break}
    \Else
    \State $\tau=\max\{\tau_\mathrm{min},\gamma_\tau\tau\}$
    \EndIf
    \EndWhile
    \If{$\|q^{l+1}-q^{l}\|<\varepsilon_\tau$}
    \State\textbf{break}
    \EndIf
    \EndFor
    \end{algorithmic}
\end{algorithm}
\begin{algorithm}
\caption{Robust Topology Optimization Algorithm with Level Set and Topological Derivative.}
\label{alg:RobustLevelset}
    \underline{Offline phase}
    \begin{algorithmic}
        \State \textbf{Choose} samples $0=t_1<t_2<...<t_K=t_\mathrm{max}, q_1,...,q_L\in U$
        \For {$l=1,...,L$}
        \For {$k=1,...,K$}
        \State \textbf{Solve} exterior problem \eqref{eq:exterior} for $U=t_ke_0, q=q_l$
        \State \textbf{Evaluate} $f_{1,l,k}^{f\rightarrow a},f_{2,l,k}^{f\rightarrow a},f_{1,l,k}^{a\rightarrow f},f_{2,l,k}^{a\rightarrow f}$ considering $h=h(b,q_l)$\eqref{eq:TDsample}
        \EndFor
        \EndFor
    \State \textbf{Compute} interpolation $\tilde{f}_{1}^{f\rightarrow a},\tilde{f}_{2}^{f\rightarrow a},\tilde{f}_{1}^{a\rightarrow f},\tilde{f}_2^{a\rightarrow f}$ \eqref{eq:TDinterpolation}
    \end{algorithmic}
    \underline{Online phase}
\begin{algorithmic}
    \State\textbf{Choose} $\psi_0\in\mathcal{S}, k_\mathrm{max}<\infty,\varepsilon>0, 0<s_\mathrm{min}<s_\mathrm{max}\le 1, \gamma\in(0,1), \delta\ge1$
    \For{$k=0,...,k_\mathrm{max}$}
    \State \textbf{Find} $q^*_k$ by Algorithm \ref{alg:InnerProblem} for $\Omega=\Omega_{\psi_k}$
    \State \textbf{Compute} $g_{k}=g_{\Omega_{\psi_k},q^*_k}$ (solve adjoint equation \eqref{eq:adjoint}, evaluate \eqref{eq:TDevaluate}\eqref{eq:TDinterpolation} considering $q^*_k$, apply \eqref{eq:TDorientedRobust})
    \If{$\theta_k=\arccos\frac{\left(\psi_k,g_{k}\right)_{L_2(D)}}{\|g_{k}\|_{L_2(D)}}<\varepsilon$}
    \State\textbf{break}
    \EndIf
    \State $s=s_\mathrm{max}$
    \While{$s> s_\mathrm{min}$}
    \State $\psi_{k+1}=\frac{1}{\sin\theta_k}\left(\sin((1-s)\theta_k)\psi_k+\sin(s\theta_k)\frac{g_{k}}{\|g_{k}\|_{L_2(D)}}\right)$
    \State \textbf{Find} $q^*_{k+1}$ by Algorithm \ref{alg:InnerProblem} for $\Omega=\Omega_{\psi_{k+1}}$
    \If {$\J(\Omega_{\psi_{k+1}},q^*_{k+1})<\J(\Omega_{\psi_k},q^*_k)$} (evaluate \eqref{eq:reducedPROBLEM} by solving \eqref{eq:magnetostatic})
    \State{$s=\min\{s_\mathrm{max},\delta s\}$}
    \State\textbf{break}
    \Else
    \State $s=\max\{s_\mathrm{smin},\gamma s\}$
    \EndIf
    \EndWhile
    \EndFor
\end{algorithmic}
\end{algorithm}
To determine the worst case $q^*(\Omega)$ we solve the constrained maximization problem by a projected gradient method stated in Algorithm~\ref{alg:InnerProblem}. The operator $\mathcal{P}_U$ denotes the orthogonal projection onto the uncertainty set $U$ and the gradient $\nabla_q\J$ is computed by
\begin{align*}
    \nabla_q\J(\Omega,q)=\nabla_q\mathcal{G}(\Omega,q,(\alpha^1,...,\alpha^N),(u^1,...,u^n,\lambda^1,...,\lambda^n),(p^1,...,p^n,\eta^1,...,\eta^n))
\end{align*}
where $\mathcal{G}$ is the Lagrangian defined in \eqref{eq:Lagrangian}, $(u^n,\lambda^n)$ solves the state equation \eqref{eq:magnetostatic} and $(p^n,\eta^n)$ the adjoint equation \eqref{eq:adjoint} for $n=0,...,N-1$.
\begin{remark}
    Let us state this expression for the specific uncertainties which we are considering. Both are incorporated via the PDE constraint \eqref{eq:magnetostatic}. We consider uncertainties in the source current $j(\alpha,q)$ and the material law $h_\Omega(b,q)$. The gradient of $\J$ is then computed by
    \begin{align}
    \label{eq:gradientq}
        \nabla_q\J(\Omega,q)=\sum_{n=0}^{N-1}\int_{D_\mathrm{all}}\nabla_qh_\Omega(\cur u^n,q)\cdot\cur p^n\dx-\int_{D_\mathrm{all}}\nabla_qj(\alpha^n,q)p^n\dx.
    \end{align}
\end{remark}
\begin{remark}
In order to apply Theorem \ref{thm:robust} it is necessary that $q^*(\Omega)$ is indeed a global maximizer. We try to guarantee this by choosing the initial guess $q^0$ wisely, because the quality of the solution of the local maximization problems depends on the starting point $q^0$. If \( U = [u_1, u_2] \subset \mathbb{R} \) is an interval, we use \( q^0 \in \{u_1, u_2\} \), the boundary points of \( U \), as initial guesses for the inner maximization problem. This approach is motivated by the observation that promising candidates for the maximizer \( q^* \) for such physical problems often lie on the boundary of \( U \). However, care must be taken if \( U \subset \mathbb{R}^n \), as selecting an inappropriate boundary point may lead to suboptimal solutions and fail to identify the worst case. For further details we refer the interested reader to \citep{BenTal2021BeyondLO}.    
\end{remark}
\begin{remark}
    If the uncertain parameter $q$ affects the material law $h_\Omega(b,q)$ we have to precompute the topological derivative \eqref{eq:TDsample} also for samples $q_0, \dots, q_L\in U$. This is done in the offline phase of Algorithm~\ref{alg:RobustLevelset}. Depending on its dimension this can get computationally costly. However, since these precomputations are independent of each other one can do this in parallel. If $q$ acts only on $j$ this is not necessary and the loop over $l=0,...,L$ in the offline phase of Algorithm~\ref{alg:RobustLevelset} can be skipped.
\end{remark}
\begin{remark}
    Compared to the non-robust optimization algorithm we have to additionally solve the inner maximization problem which is computationally costly. The computational overhead depends on the initial guesses and the regularity with respect to the parameter. In Section~\ref{sec: Numerical Results} we present a comparison of computation time for nominal and robust optimization for our applications, see Table \ref{tab:time}.
\end{remark}

\section{Numerical Results}
\label{sec: Numerical Results}
In this section we present results of the robust design optimization problem \eqref{eq:robustPROBLEM} applied to an electric machine with uncertainty either in the application current or the material law. In the latter case we distinguish between a scalar parameter or a spatially distributed one. We compare design and worst case values with those obtained by a nominal optimization. An overview of the obtained performances is given in Table \ref{tab:comparison}. We want to point out that our approach is quite general and can be applied not only to electric machines, but to many other engineering problems. 
% The choice of parameters and the size of the uncertainty in the following results are a physically motivated proof of concept. 
\subsection{Implementation and Parameters}
In our simulations, we consider \eqref{eq:magnetostatic} with the piecewise defined material law \eqref{eq:hpiecewise} based on the following nominal material laws 
\begin{align}\label{eq:bhcurve}
\begin{split}
    h_a(b)&=\nu_0b,\\
    h_{m_i}(b)&=\nu_m(b-b_Re_{\varphi_i}), i=1,2,\\
    h_f(b)&=\nu_0b+(\nu_f-\nu_0)\frac{K_f}{\sqrt[N_f]{K_f^{N_f}+|b|^{N_f}}}b
\end{split}
\end{align}
and a source current density
\begin{align}\label{eq:sourcecurrent}
\begin{split}
    j_A(\alpha)&=\hat{j}\sin(4\alpha+\phi_0)\\
    j_B(\alpha)&=\hat{j}\sin(4\alpha +\frac{2\pi}{3}+\phi_0)\\
    j_C(\alpha)&=\hat{j}\sin(4\alpha +\frac{4\pi}{3}+\phi_0)\\
\end{split}
\end{align}
with parameters given in Table \ref{tab:mat_params}.
\begin{table}
    \centering
    \begin{tabular}{c|c|c|c|c|c|c|c|c|c}
         $\nu_0$&$\nu_f$&$K_f$&$N_f$&$\nu_m$&$b_R$&$\varphi_1$&$\varphi_2$&$\hat{j}$&$\phi_0$\\\hline
         $\frac{10^7}{4\pi}$&200&2.2T&12&$\frac{\nu_0}{1.086}$&1.216T&$30^\circ$&$15^\circ$&$23.7\times10^6\tfrac{\mathrm{A}}{\mathrm{m}^2}$&$6^\circ$
    \end{tabular}
    \caption{Material parameters.}
    \label{tab:mat_params}
\end{table}
To solve the occurring PDEs \eqref{eq:magnetostatic},\eqref{eq:exterior} we used the open source finite element package NGSolve \citep{ngsolve}. Similarly as in \citep{Amstutz_2006ab} we used lowest order finite elements on a triangular mesh with 3995 nodes to solve the state \eqref{eq:magnetostatic} and adjoint equation \eqref{eq:adjoint} as well as to represent the level set function \eqref{eq:levelset}. The goal is to minimize  the negative average torque \eqref{eq:reducedPROBLEM} for $N=11$ rotor positions \eqref{eq:rotorpositions}. To compute the topological derivative \eqref{eq:TDformula}, which includes solving the auxiliary problem \eqref{eq:exterior}, we truncated the unbounded domain $\R^2$ to a ball with radius 128 and used lowest order finite elements on a triangular mesh with 61272 nodes. The nominal optimization was done using Algorithm \ref{alg:Levelset}, for the robust optimization we used Algorithms \ref{alg:InnerProblem}, \ref{alg:RobustLevelset}. The parameters in the level set algorithms were chosen as presented in Table \ref{tab:param_nominal}. The samples in the offline phase of Algorithm \ref{alg:Levelset} are computed for uniformly distributed $0=t_1<...<t_K=t_\mathrm{max}$. If the uncertainty $q$ enters the material law as considered in Section~\ref{subsec:bh}, we additionally need $q_1,...,q_L\in U$ in the offline phase of Algorithm \ref{alg:RobustLevelset}, which we also chose uniformly distributed.
\begin{table}
    \centering
    \begin{tabular}{c|c|c|c|c|c|c|c|c|c|c}
        $t_\mathrm{max}$ &$K$&$L$&$\varepsilon$&$s_\mathrm{min}$&$s_\mathrm{max}$&$\gamma,\gamma_\tau$&$\delta,\delta_\tau$&$\varepsilon_\tau$&$\tau_\mathrm{min}$&$\tau_\mathrm{max}$ \\\hline
         5& 50&10&$2^\circ$&0.05&1&0.5&1.5&$10^{-3}$&$10^{-3}$&1
    \end{tabular}
    \caption{Parameter values for Algorithm \ref{alg:Levelset}, \ref{alg:InnerProblem}, \ref{alg:RobustLevelset}.}
    \label{tab:param_nominal}
\end{table}

\subsection{Nominal Optimization}\label{subsec:NOM}
We apply Algorithm \ref{alg:Levelset} to find a solution to the optimization problem without uncertainty \eqref{eq:PROBLEM}. After 29 iterations, the level set algorithm converged and we obtained the design $\Omega_\mathrm{NOM}$ shown in Figure \ref{fig:nominal} with a negative average torque of $\J(\Omega_\mathrm{NOM})=-835\mathrm{Nm}$.
\begin{figure}
    \centering
    \includegraphics[width=0.7\textwidth,trim=20cm 7.3cm 20cm 7.3cm,clip]{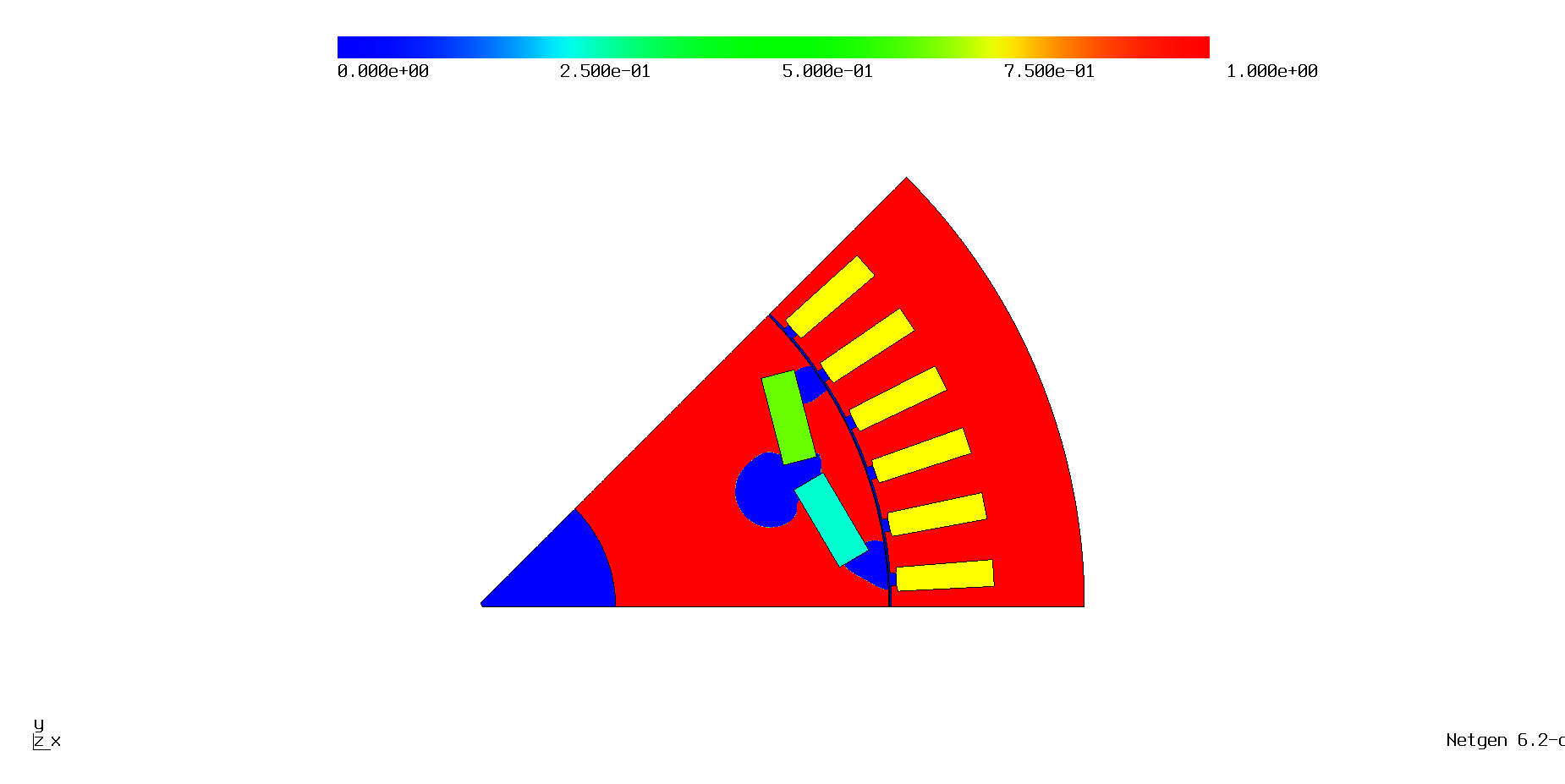}
    \caption{Final design of nominal optimization $\Omega_\mathrm{NOM}$, $\J(\Omega_\mathrm{NOM})=-835\textrm{Nm}$.}
    \label{fig:nominal}
\end{figure}
\subsection{Uncertain Load Angle}\label{subsec:ANG}
As a first application of the developed robust design optimization algorithm we add uncertainty to the source current density \eqref{eq:sourcecurrent}, which is the right hand side of the magnetostatic PDE \eqref{eq:magnetostatic},
\begin{align*}
    j_A(\alpha,q)&=\hat{j}\sin(4\alpha +q)\\
    j_B(\alpha,q)&=\hat{j}\sin(4\alpha +\frac{2\pi}{3}+q)\\
    j_C(\alpha,q)&=\hat{j}\sin(4\alpha +\frac{4\pi}{3}+q),\\
\end{align*}
where $q$ changes the phasing of the electric source current density with respect to the mechanical rotor position $\alpha$. This parameter is called the load angle. The uncertainty set is chosen as $U=[-9^\circ,21^\circ]$ with nominal value $\hat{q}=6^\circ$. 
\begin{figure}
    \centering
    \includegraphics[width=0.45\textwidth,trim=20cm 7.3cm 20cm 7.3cm,clip]{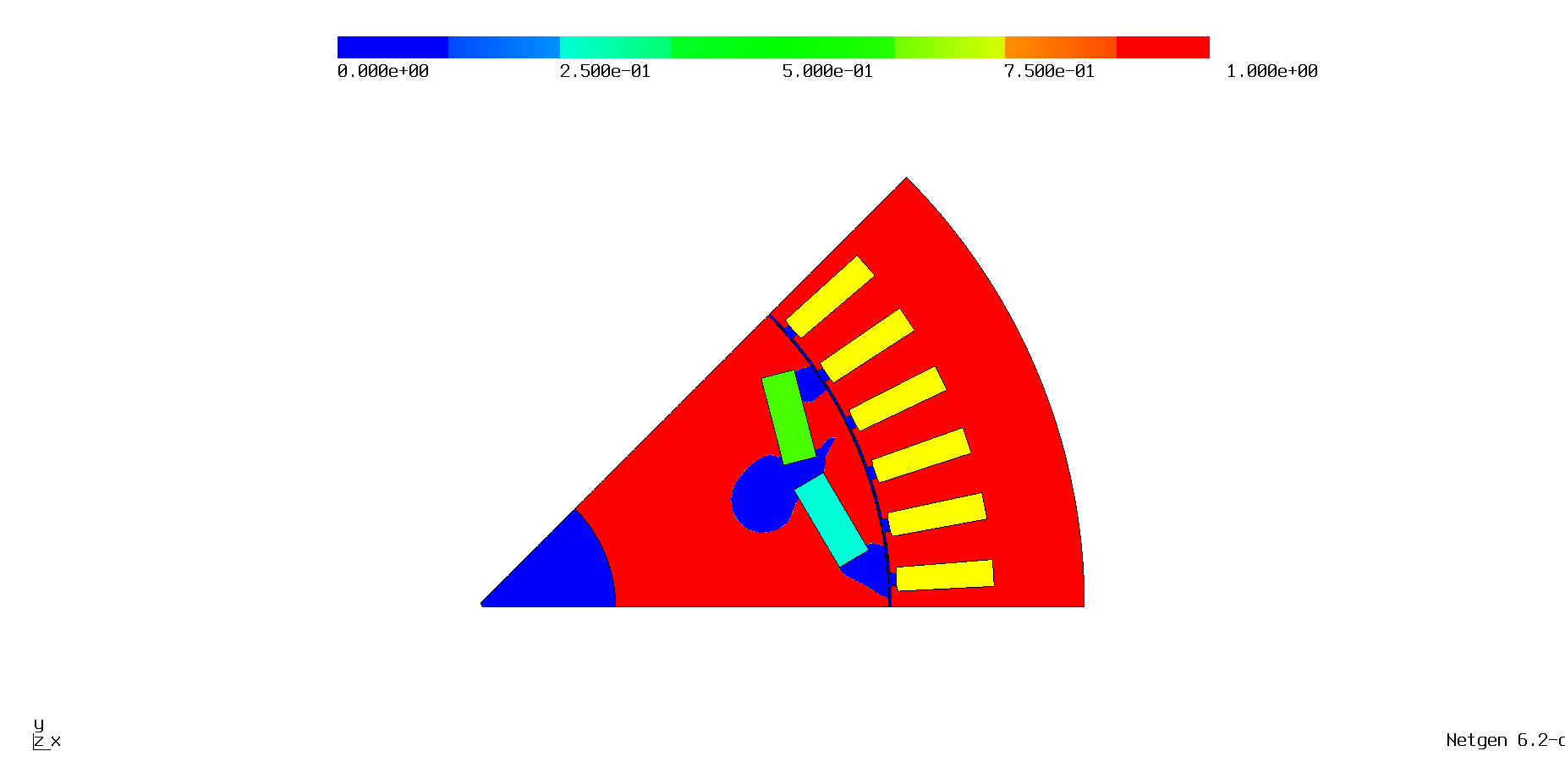}\quad
    \includegraphics[width=0.45\textwidth,trim=20cm 7.3cm 20cm 7.3cm,clip]{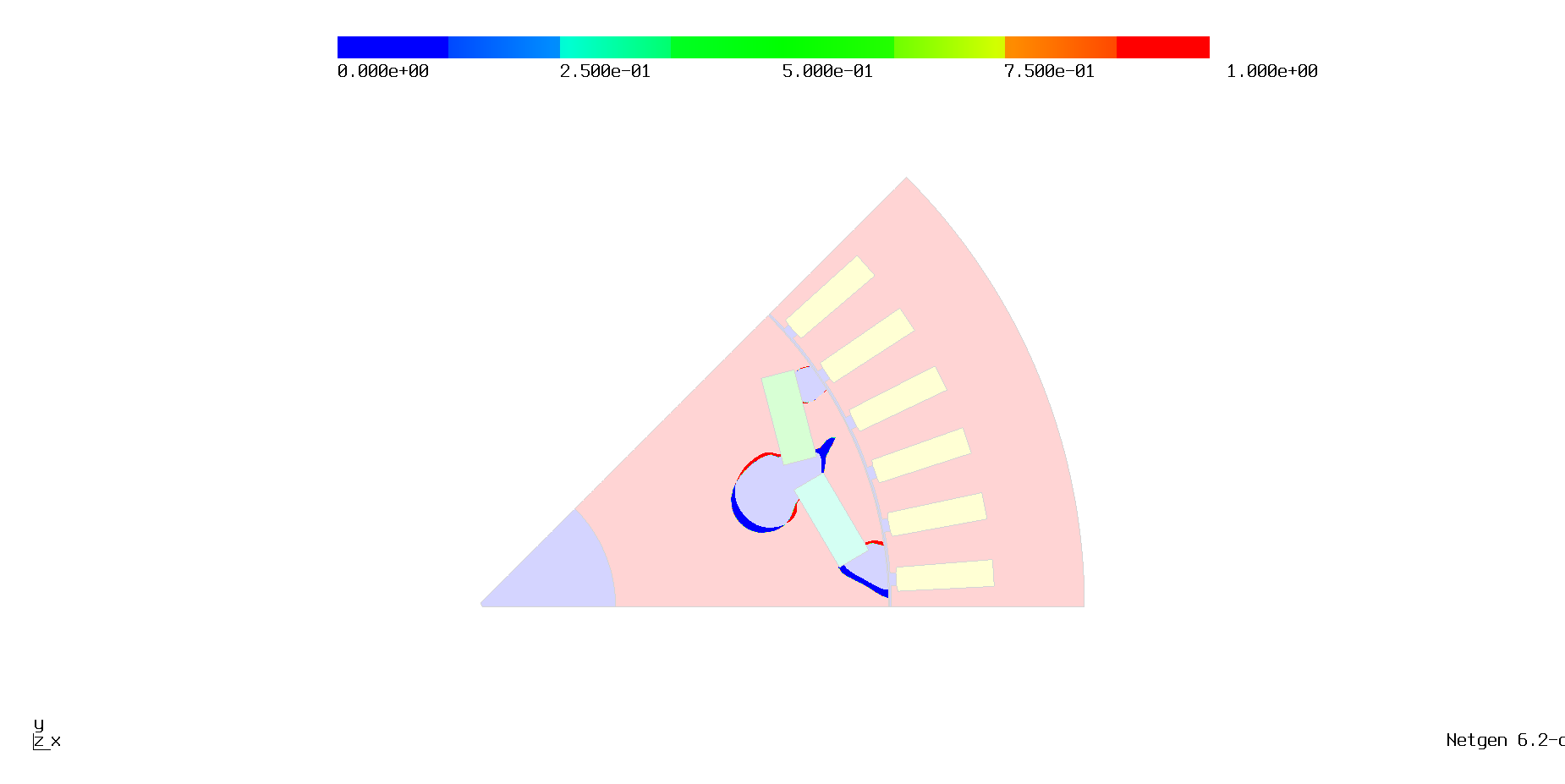}
    \caption{Final design $\Omega_\mathrm{ANG}$ (left), difference from nominal result $\Omega_\mathrm{NOM}$ (right) considering an uncertain load angle, see Section~\ref{subsec:ANG}.}
    \label{fig:design_angle}
\end{figure}
\begin{figure}
    \centering
    \includegraphics[width=0.7\textwidth]{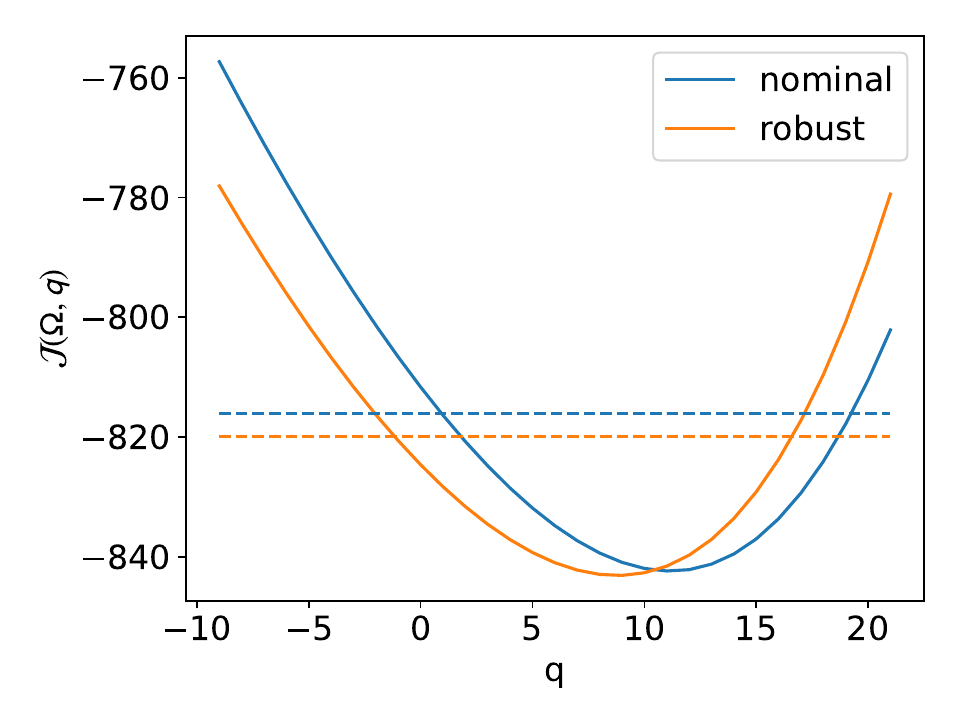}
    \caption{Negative average torque in dependence of $q\in U$ for nominally (blue) and robustly (orange) optimized design with corresponding averages over $U$ (dashed) as considered in Section~\ref{subsec:ANG}.}
    \label{fig:plot_angle}
\end{figure}

In Figure \ref{fig:design_angle} we present the robustly optimized design $\Omega_\mathrm{ANG}$ obtained after 76 iterations of Algorithm \ref{alg:RobustLevelset} and the difference to the nominally optimized one, i.e. with $\hat{q}=6^\circ$. The plot of the performance in dependence of $q\in U$ shows that the worst cases are attained at $q^*=-9^\circ$ for both designs with worst case values ${\J(\Omega_\mathrm{NOM},q^*(\Omega_\mathrm{NOM}))=-757\mathrm{Nm}},{\J(\Omega_\mathrm{ANG},q^*(\Omega_\mathrm{ANG}))=-778\mathrm{Nm}}$. This is an improvement of $3\%$. Also the nominal value is slightly improved from $\J(\Omega_\mathrm{NOM},\hat{q})=-835\mathrm{Nm}$ to $\J(\Omega_\mathrm{ANG},\hat{q})=-841\mathrm{Nm}$ which is $0.7\%$ as well as the average performance over the uncertainty set $U$ from $-816\mathrm{Nm}$ to $-820\mathrm{Nm}$. By considering the robustness, we obtained a better local minimizer also for the nominal case. This can be explained by a flattening of the objective, as illustrated in the introductory example Figure \ref{fig:RobustOptimizationPlot}.
\begin{remark}
    We mention that, in this case, the optimality criterion \eqref{eq:lemmacondition} was not reached. 
    % Since this condition is only sufficient but not necessary for local optimality according to Definition \ref{def:Optimality}, there can be optimal designs which do not satisfy the criterion. This was also observed in, e.g., \citep{Gangl_2020ae}. 
    In fact, the algorithm fell into a cyclic behavior of three designs after approximately 60 iterations. We stopped the algorithm after $k_{\mathrm{max}}=100$ iterations and selected the best performing design.
\end{remark}

\subsection{Uncertain Material Parameter}\label{subsec:bh}
We consider an uncertainty in the material law of iron. Using the model \eqref{eq:bhcurve} we model an uncertain saturation behavior by
\begin{align*}
    h_f(b,q)=\nu_0b+(\nu_f-\nu_0)\frac{q}{\sqrt[N_f]{q^{N_f}+|b|^{N_f}}}b,
\end{align*}
with $q\in U=[1.76,2.64]$ which corresponds to a $\pm20\%$ uncertainty of the nominal value $\hat{q}=K_f=2.2$, displayed in Figure \ref{fig:bh_curves}.
\begin{figure}
    \centering
    \includegraphics[width=0.7\textwidth]{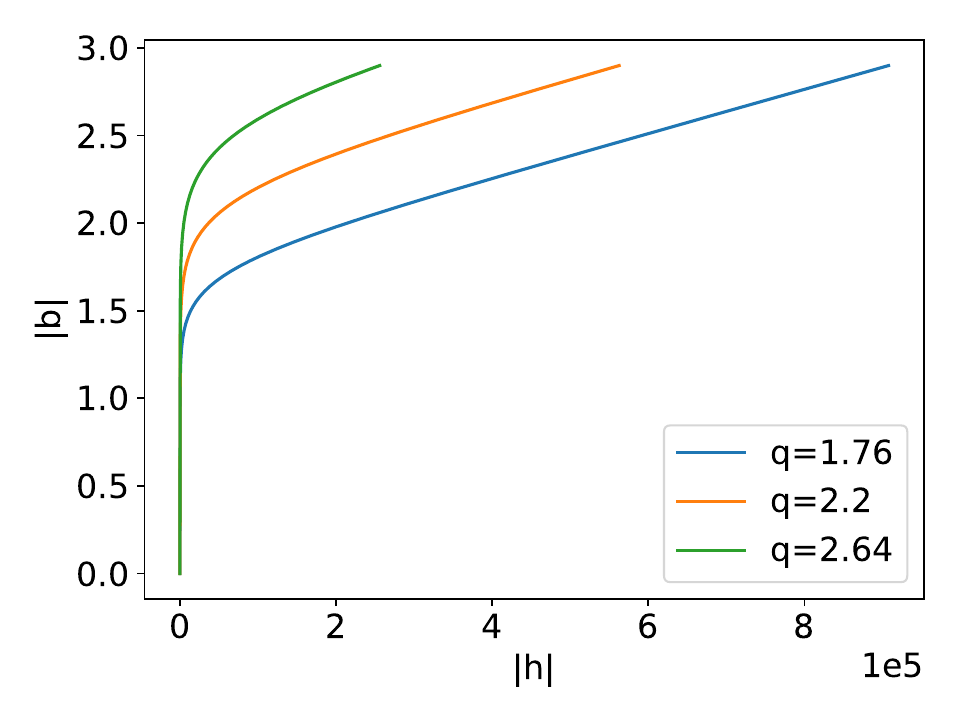}
    \caption{Iron material law for $q\in\{1.76,2.2,2.64\}.$}
    \label{fig:bh_curves}
\end{figure}
\subsubsection{Scalar Parameter}\label{ssubsec:SCAL}
\begin{figure}
    \centering
    \includegraphics[width=0.45\textwidth,trim=20cm 7.3cm 20cm 7.3cm,clip]{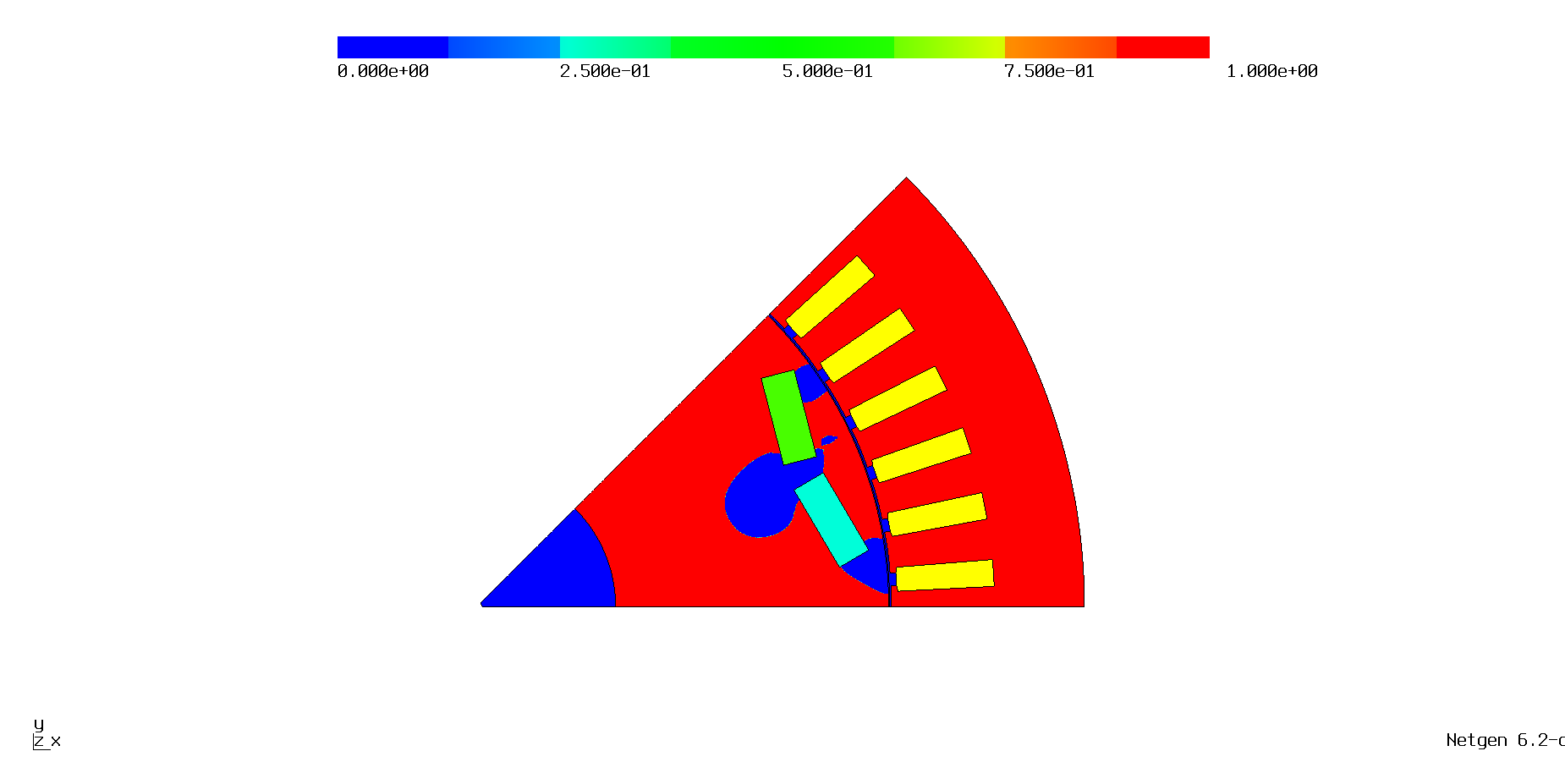}\quad
    \includegraphics[width=0.45\textwidth,trim=20cm 7.3cm 20cm 7.3cm,clip]{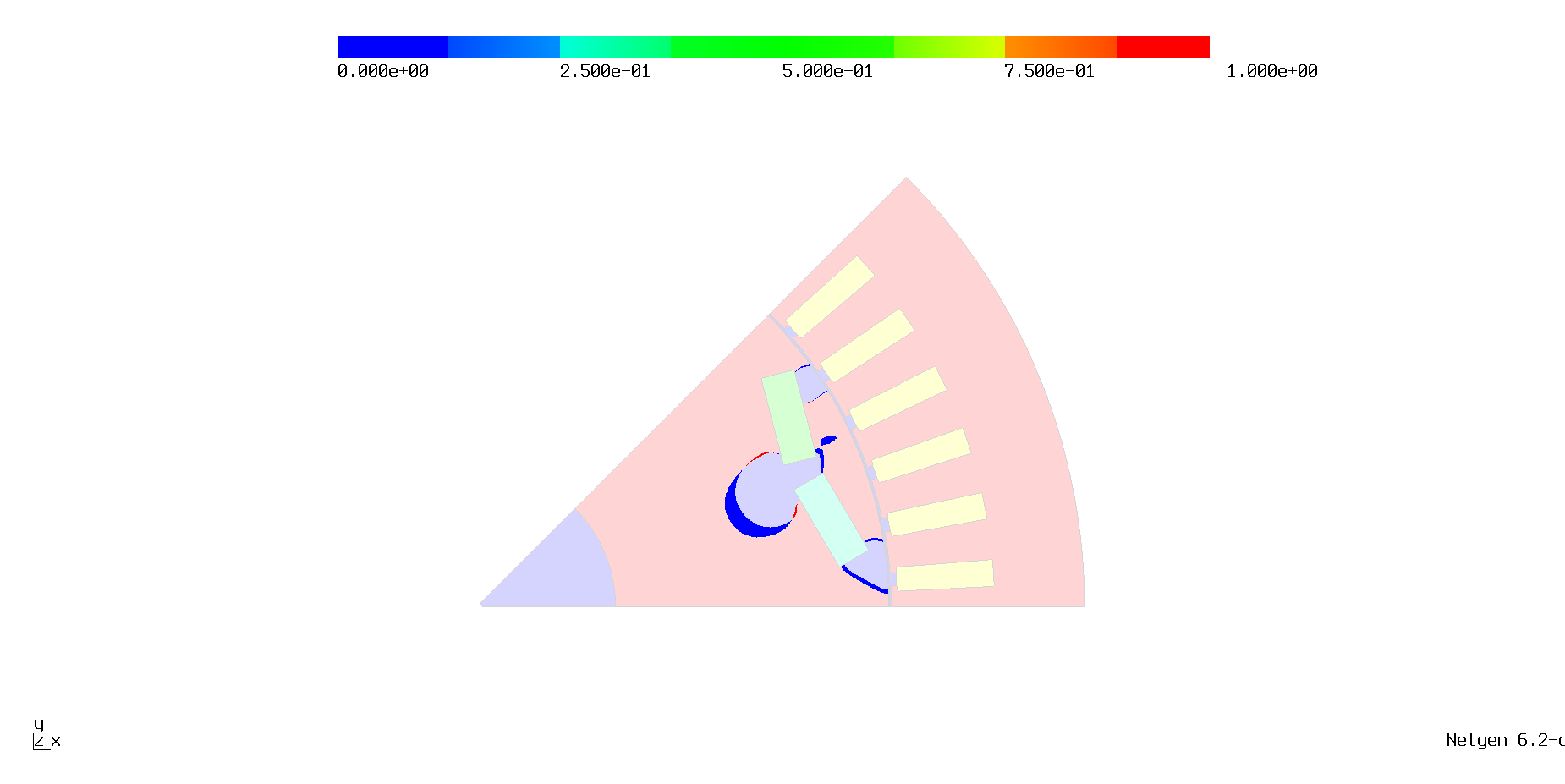}
    \caption{Final design $\Omega_\mathrm{SCAL}$ of robust optimization with uncertain scalar material parameter (left) and difference from the nominal result $\Omega_\mathrm{NOM}$ (right) considering a scalar material uncertainty, see Section~\ref{ssubsec:SCAL}.}
    \label{fig:design_bh_scalar}
\end{figure}
\begin{figure}
    \centering
    \includegraphics[width=0.7\textwidth]{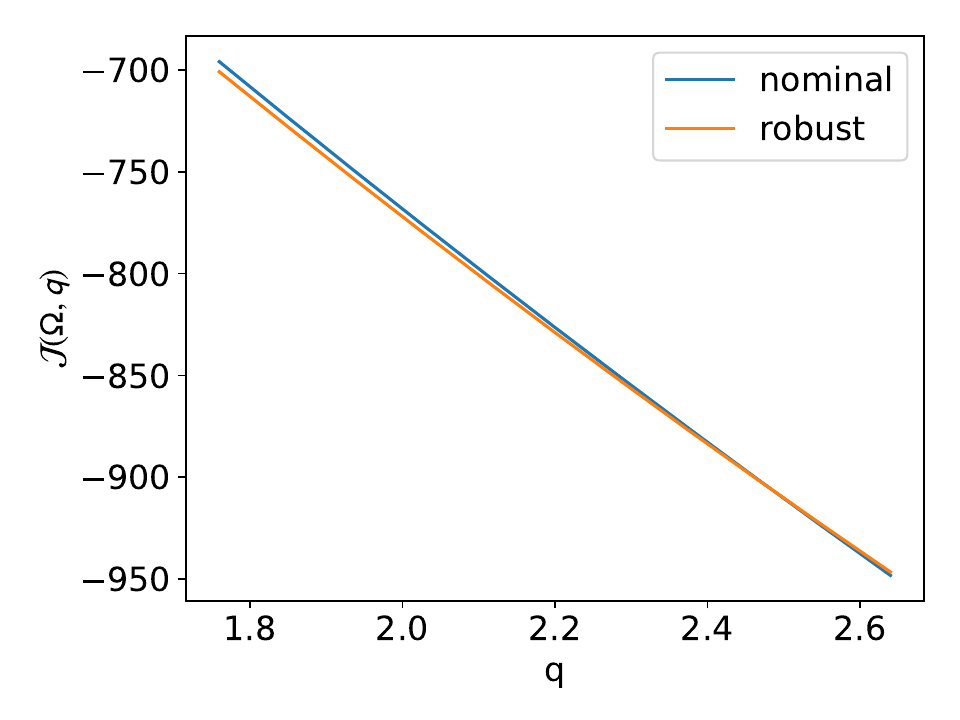}
    \caption{Negative average torque in dependence of $q\in U$ for nominally (blue) and robustly (orange) optimized design, as considered in Section~\ref{ssubsec:SCAL}.}
    \label{fig:plot_bh_scalar}
\end{figure}
First we assume that the uncertainty is homogeneous for all iron parts modeled. This corresponds to a scalar parameter $q\in U.$
The robust design optimization using Algorithm \ref{alg:RobustLevelset} converged after 30 iterations. The resulting design $\Omega_\mathrm{SCAL}$ is presented in Figure \ref{fig:design_bh_scalar}. The worst case value of the cost function decreases from $\J(\Omega_\mathrm{NOM},q^*(\Omega_\mathrm{NOM}))=-696\mathrm{Nm}$ to $\J(\Omega_\mathrm{SCAL},q^*(\Omega_\mathrm{SCAL}))=-701\mathrm{Nm}$ corresponding to an improvement by $0.7\%$. In Figure \ref{fig:plot_bh_scalar} one can see that the worst case parameter is $q^*=1.76$ for both the nominally and the robustly optimized design. The value at the nominal parameter $\hat{q}=2.2$ is slightly improved from $\J(\Omega_\mathrm{NOM},\hat{q})=-826\mathrm{Nm}$ to $\J(\Omega_\mathrm{SCAL},\hat{q})=-829\mathrm{Nm}$. 
\subsubsection{Distributed Parameter}\label{ssubsec:DIST}
We now assume that the uncertainty of the iron material behavior varies spatially. We model this by a parameter function $q:D_\mathrm{all}\rightarrow U$ with uncertain values in $U$. The robustly optimized design $\Omega_\mathrm{DIST}$ is shown in Figure \ref{fig:design_dist} (left) together with the difference from the nominally optimized design $\Omega_\mathrm{NOM}$(right). Here, Algorithm \ref{alg:RobustLevelset} converged after 23 iterations. The worst case function values give a small improvement from $\J(\Omega_\mathrm{NOM},q^*(\Omega_\mathrm{NOM}))=-692\mathrm{Nm}$ to $\J(\Omega_\mathrm{DIST},q^*(\Omega_\mathrm{DIST}))=-696\mathrm{Nm}$. In Figure \ref{fig:param_dist} we display the spatial distribution of the worst case parameter functions $q^*(\Omega_\mathrm{DIST})$, $q^*(\Omega_\mathrm{NOM})$ for the robustly and nominally optimized designs, respectively, which are barely different. We observe that $q^*$ attains only values on the boundaries of $U=[1.76,2.64]$. Since the lower value $1.76$ highly dominates, which was also the worst case in the previous optimization in Section~\ref{ssubsec:SCAL}, it is feasible that the resulting designs and torque values are very close to each other, see Figure \ref{fig:design_bh_scalar} and Figure \ref{fig:design_dist}. In Table \ref{tab:comparison} we compare the worst case function values for the nominally and robustly optimized designs. 
\begin{remark}
    In Figure \ref{fig:param_dist} we observe a very homogeneous worst case parameter distribution in the stator $D_S$ compared to the rotor $D_R$. One can give an engineering explanation for this phenomenon: The occurring magnetic fields, coming from the permanent magnets and the excitation current, move with the same speed as the rotor. Therefore the rotor has an almost static field with small oscillations. On the other hand, the pointwise temporal average of the field in the stator is zero. Since we consider the average torque, a possible spatial variation of the uncertainty in the stator is canceled. 
\end{remark}
\begin{remark}
    When computing the topological derivative $\frac{\id^{a\rightarrow f}}{\id\Omega}\J(\Omega,q)(z)$ for $z\in\Omega_a$ one has to assign a value for $q(z)$ in the iron perturbation. This issue is often encountered in combined design and parameter optimization. The presence of air hides the effect of $q$ in $\Omega_a$. In this case we decided to set $q^*$ to the nominal value  $q^*=q^*\chi_{\Omega_f}+2.2\chi_{\Omega_a}$.
\end{remark}
\begin{remark}
    Since the spatially distributed uncertainty results in an inhomogeneous material behavior, we have to be careful when using symmetries of the machine geometry. Since every pole (in our case a piece of $45^\circ$) has the same contribution to the torque it is feasible that the worst case $q^*$ is also the same for each pole. However, to get the correct average torque we have to simulate a mechanical rotation over a full pole (different to the previous cases, where it was sufficient to consider a third of a pole, see \eqref{eq:rotorpositions}). We did so by taking $N=33$ uniform angular positions $\alpha^0,...,\alpha^{N-1}$ within $45^\circ.$
\end{remark}
\begin{figure}
    \centering
    \includegraphics[width=0.45\textwidth,trim=20cm 7.3cm 20cm 7.3cm,clip]{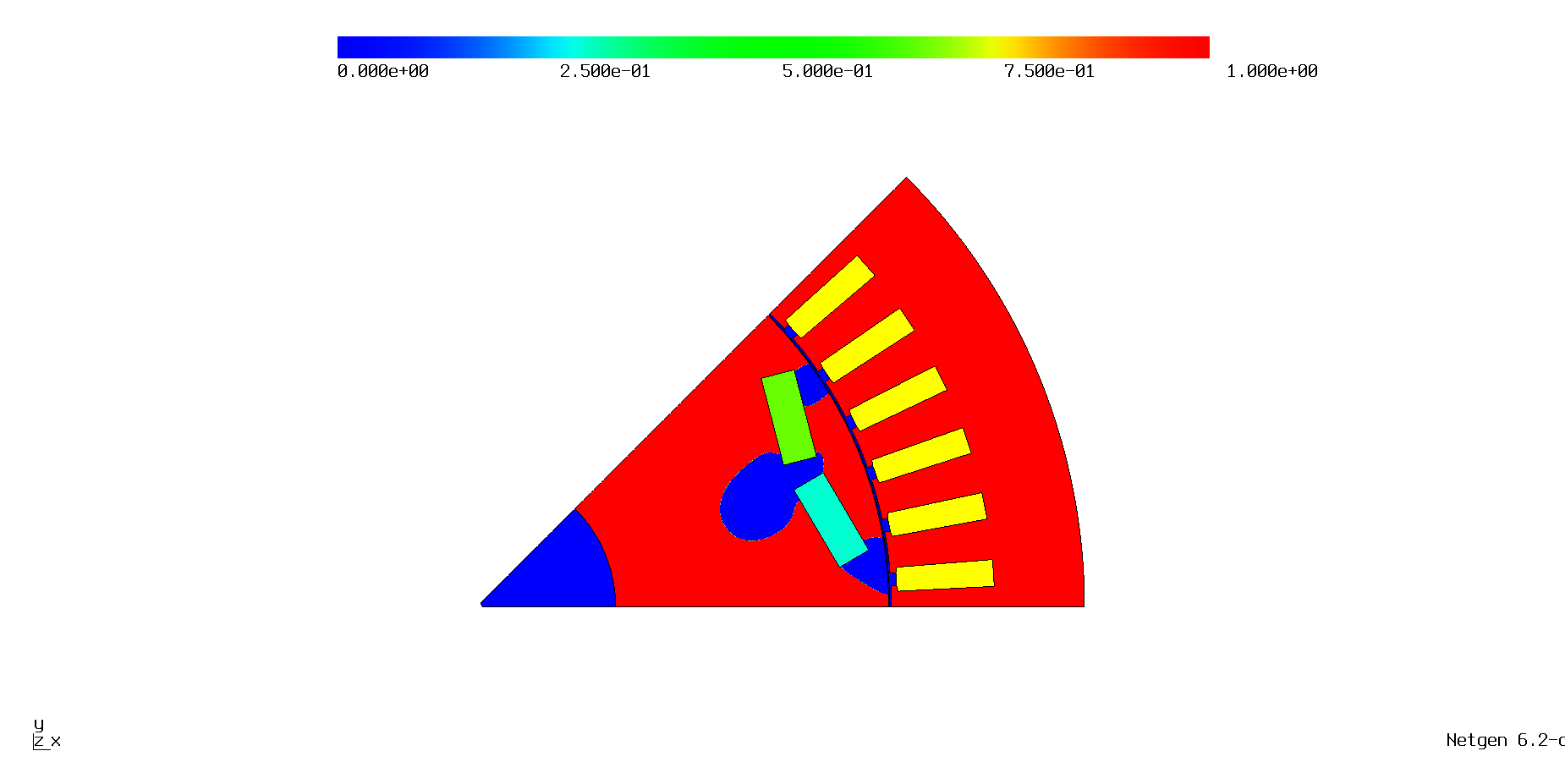}\quad\includegraphics[width=0.45\textwidth,trim=20cm 7.3cm 20cm 7.3cm,clip]{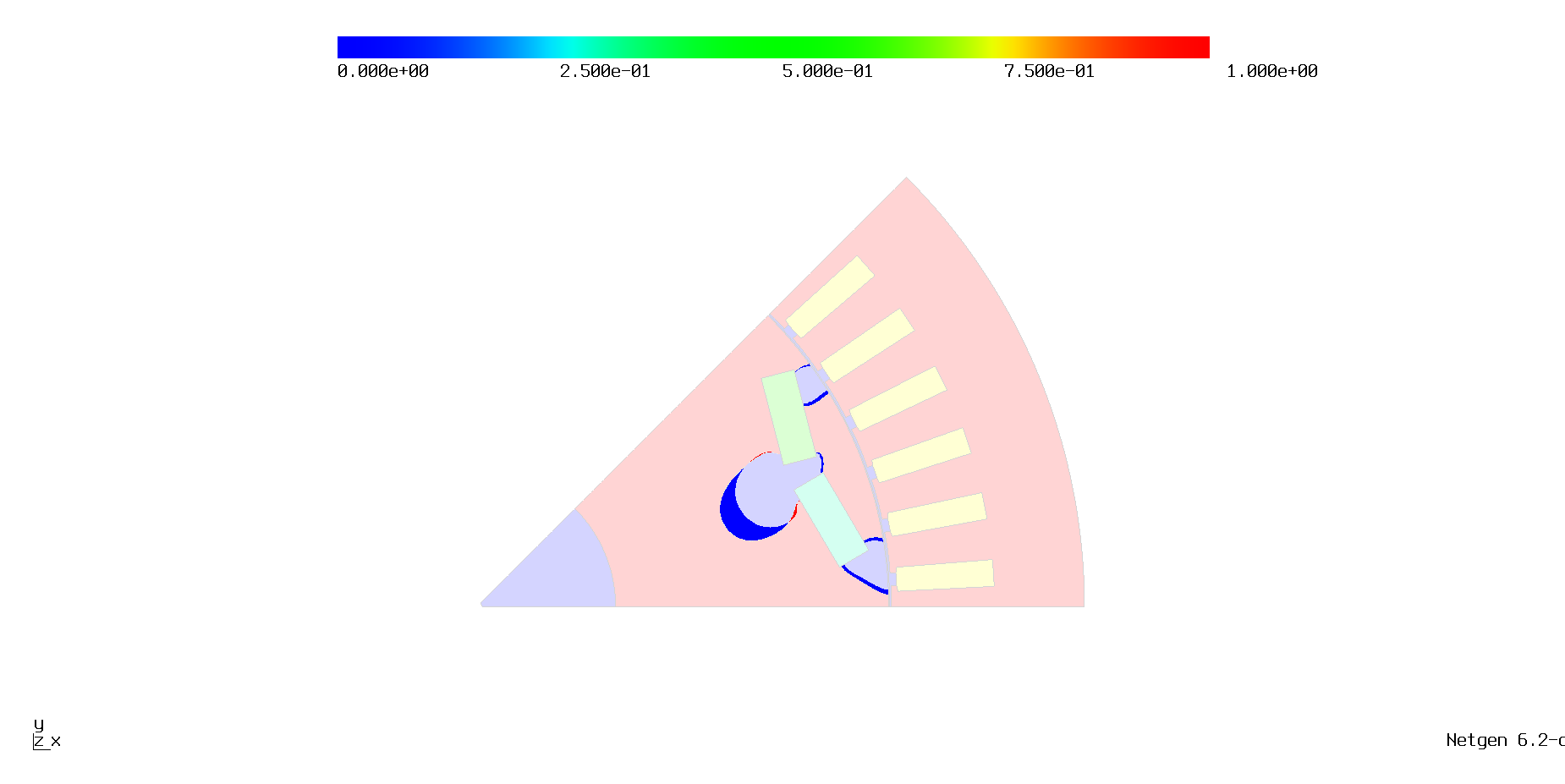}
    \caption{Final design $\Omega_\mathrm{DIST}$ for robust optimization (left), difference from the nominal result $\Omega_\mathrm{NOM}$ (right) considering a distributed material uncertainty, see Section~\ref{ssubsec:DIST}.}
    \label{fig:design_dist}
\end{figure}
\begin{figure}
    \centering
    \includegraphics[width=0.4\textwidth,trim=20cm 7.3cm 20cm 7.3cm,clip]{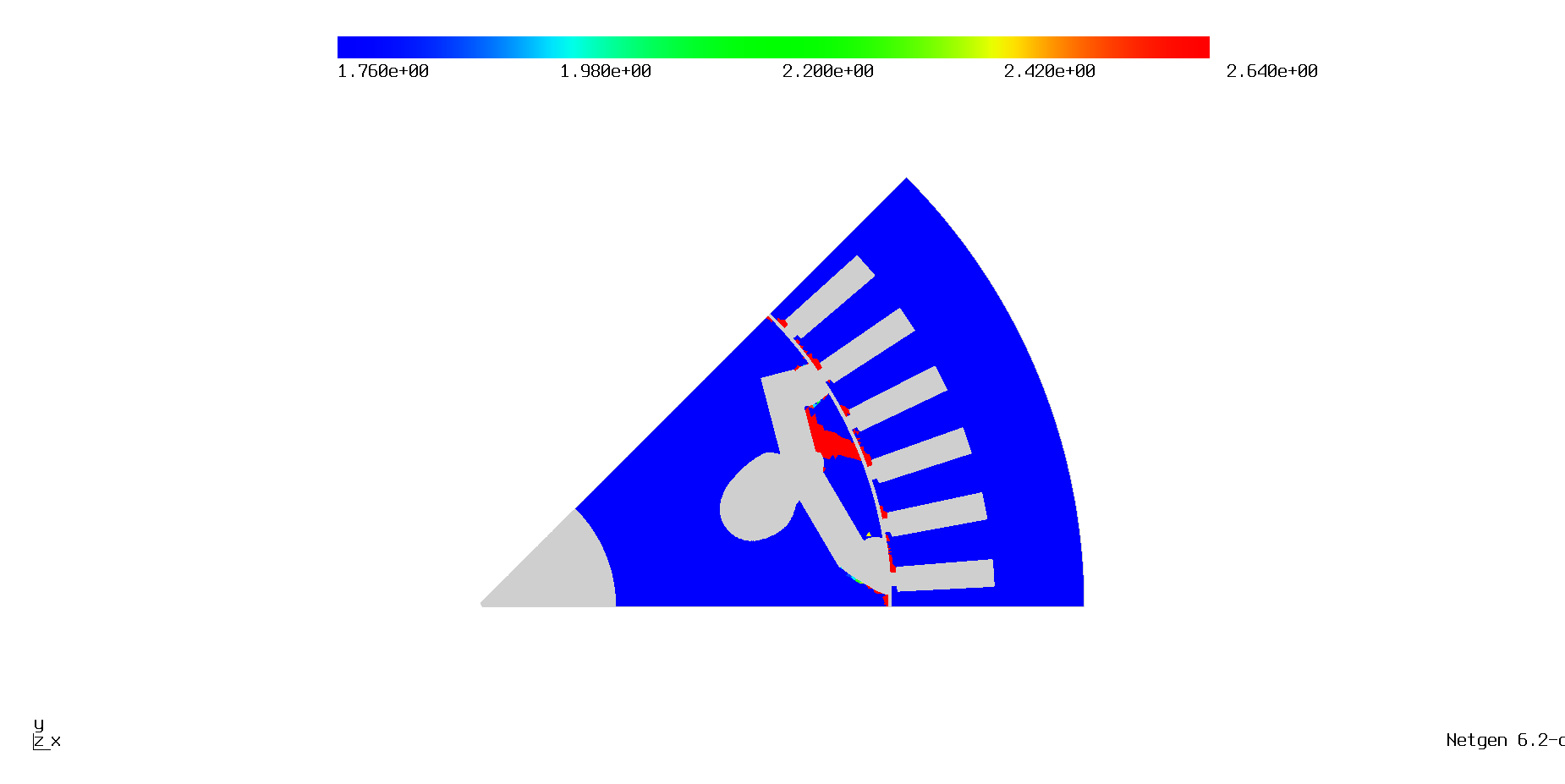}\quad\includegraphics[width=0.4\textwidth,trim=20cm 7.3cm 20cm 7.3cm,clip]{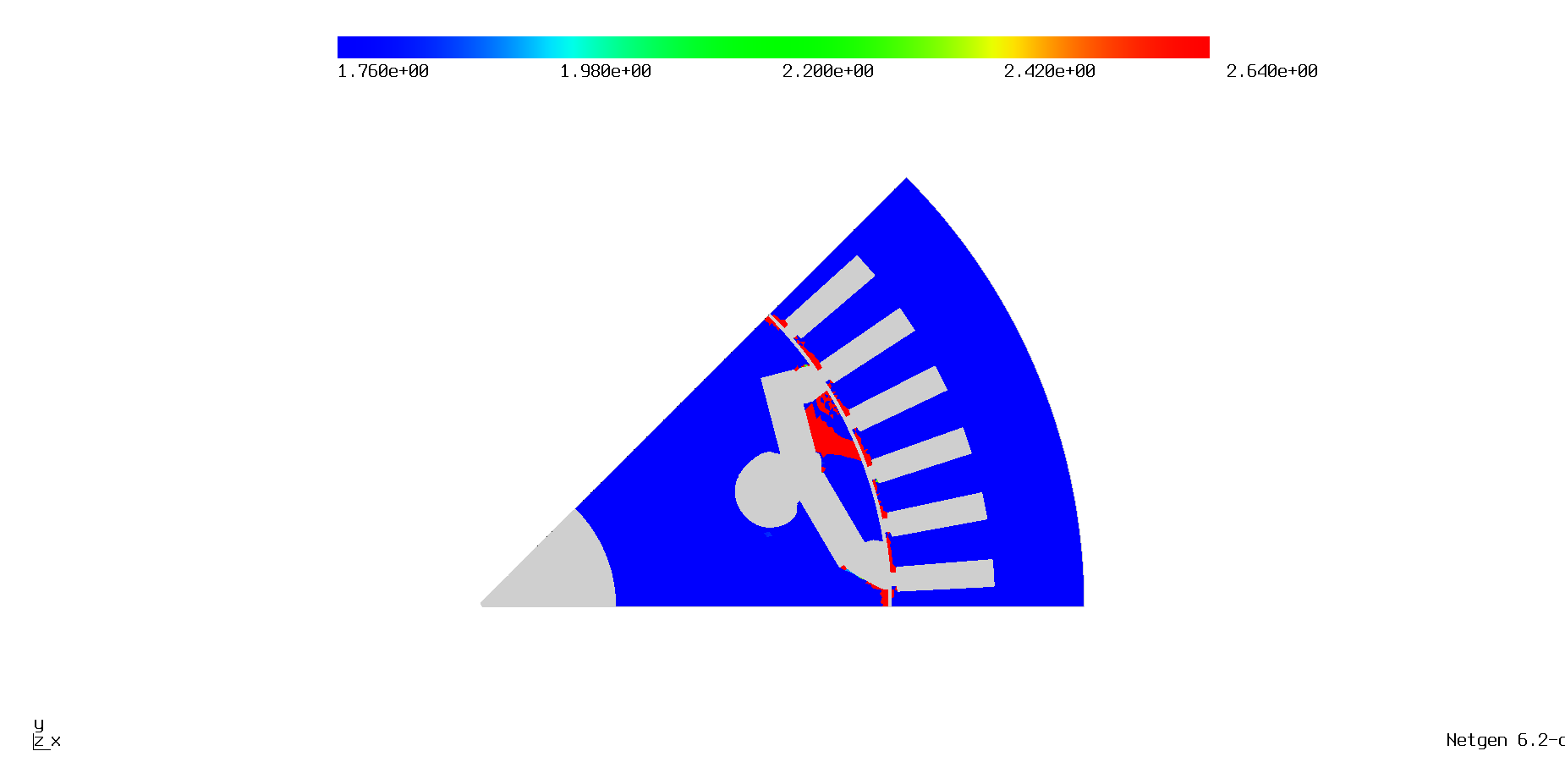}\quad\includegraphics[width=0.103\textwidth]{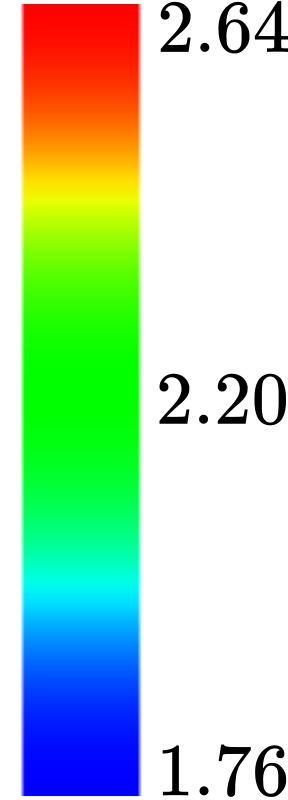}
    \caption{Worst case parameter function for final designs, $q^*(\Omega_\mathrm{DIST})$ (left), $q^*(\Omega_\mathrm{NOM})$ (right) considering a distributed material uncertainty, see Section~\ref{ssubsec:DIST}.}
    \label{fig:param_dist}
\end{figure}
\subsubsection{Single Rotor Position}\label{ssubsec:DIST1}
To highlight the effect of robust optimization we consider the torque for a single rotor position choosing $N=1$ in \eqref{eq:PROBLEM}. We obtain a new result of the nominal optimization $\Omega_\mathrm{NOM1}$ presented in Figure \ref{fig:design_nom_one}. The corresponding worst case parameter distribution $q^*(\Omega_\mathrm{NOM1})$ (Figure \ref{fig:param_dist_one}, left) is very different to the previous one (Figure \ref{fig:param_dist}, left) from Section~\ref{ssubsec:DIST}, where we considered the average torque. The robust optimization using Algorithm \ref{alg:RobustLevelset} results in a slightly different design $\Omega_\mathrm{DIST1}$, shown in Figure \ref{fig:design_dist_one}, with worst case function $q^*(\Omega_\mathrm{DIST1})$ displayed in Figure \ref{fig:param_dist_one} (right). We see a significant change in the worst case function which yields a high performance improvement from $\J(\Omega_\mathrm{NOM1},q^*(\Omega_\mathrm{NOM1}))=-870\mathrm{Nm}$ to $\J(\Omega_\mathrm{DIST1},q^*(\Omega_\mathrm{DIST1}))=-989\mathrm{Nm}$ which is $13.7\%$.

\begin{figure}
    \centering
    \includegraphics[width=0.7\textwidth,trim=20cm 7.3cm 20cm 7.3cm,clip]{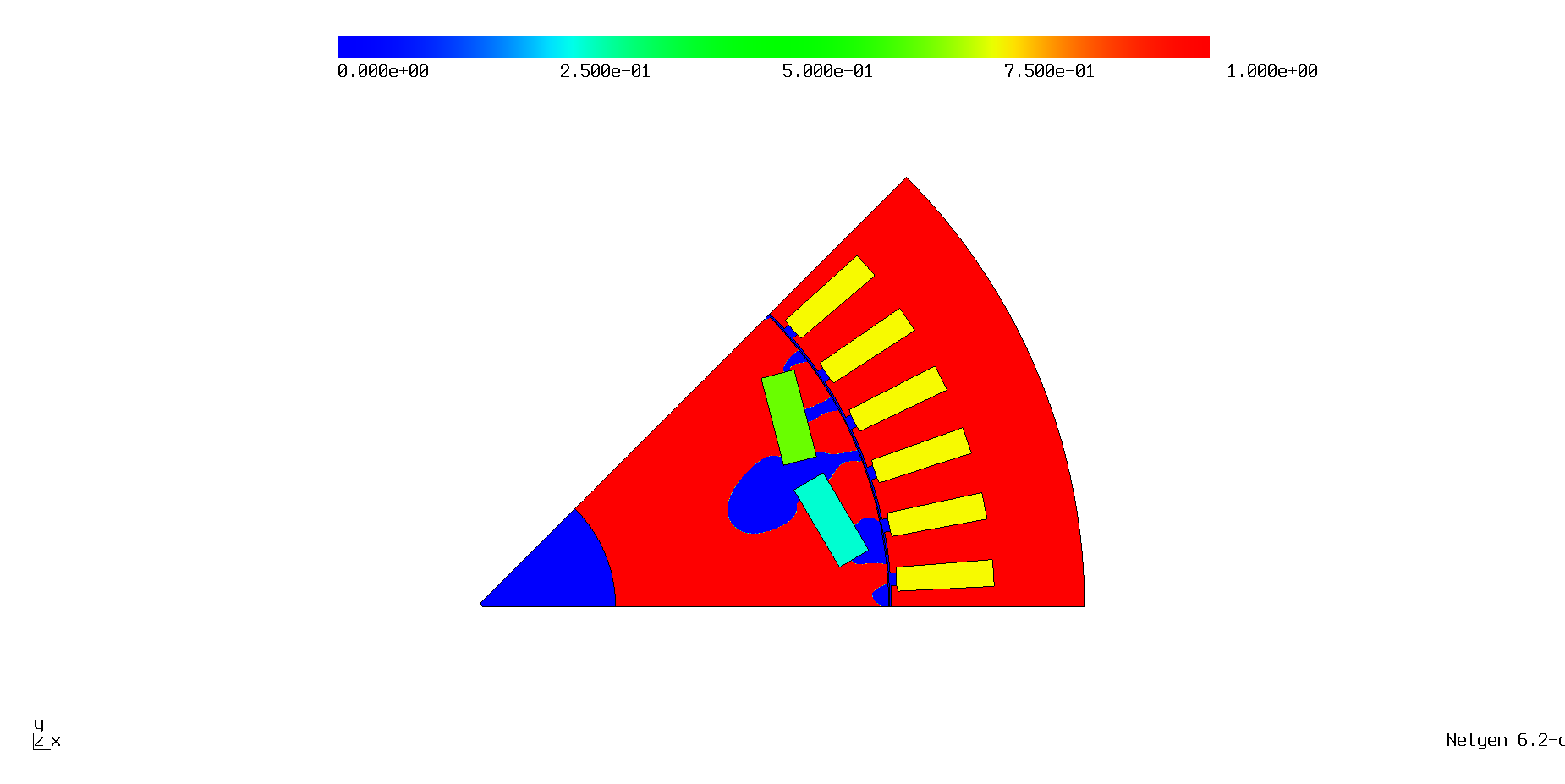}
    \caption{Final design $\Omega_\mathrm{NOM1}$ for nominal optimization considering a single rotor position, see Section~\ref{ssubsec:DIST1}.}
    \label{fig:design_nom_one}
\end{figure}
\begin{figure}
    \centering
    \includegraphics[width=0.45\textwidth,trim=20cm 7.3cm 20cm 7.3cm,clip]{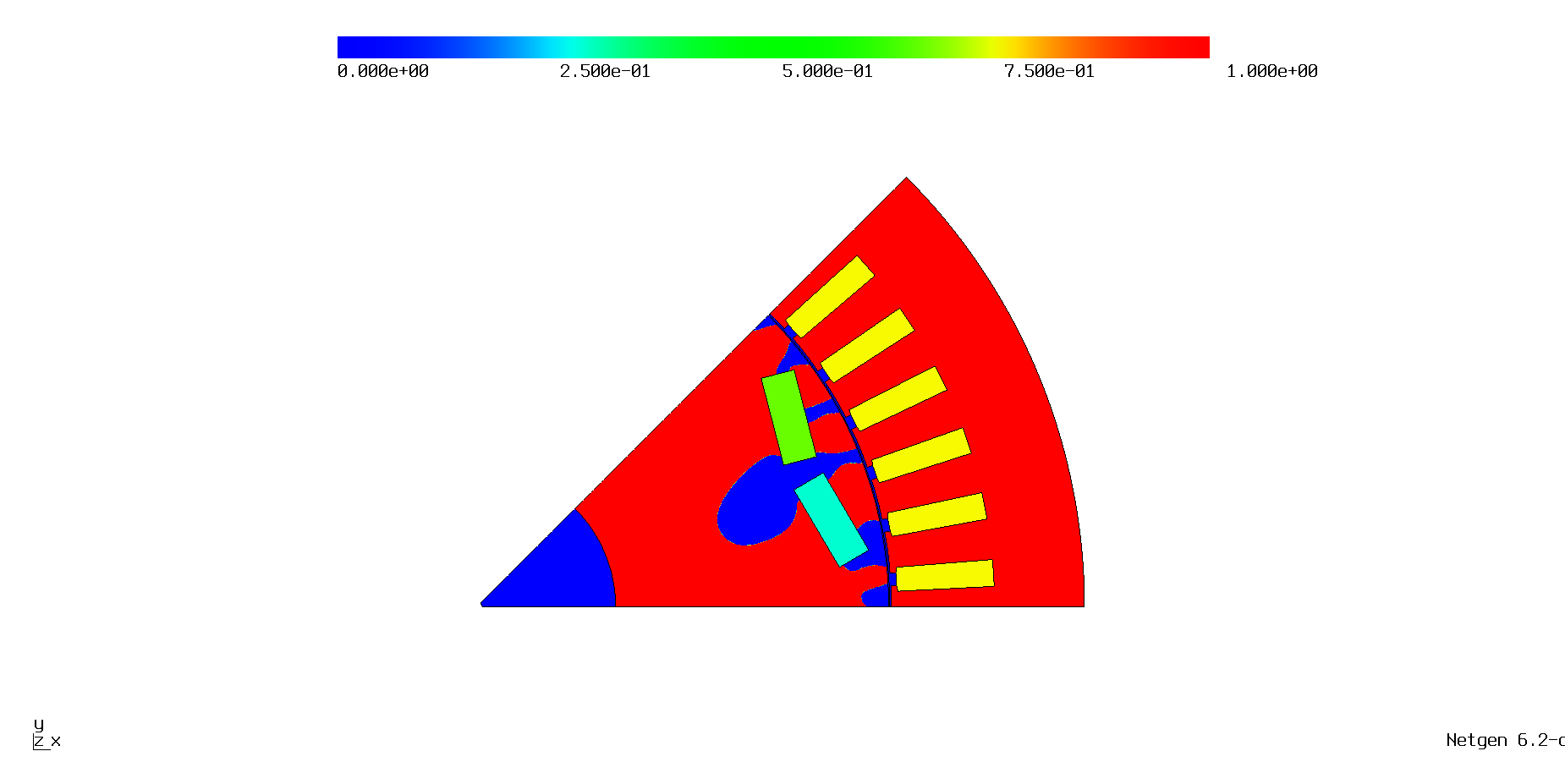}\quad\includegraphics[width=0.45\textwidth,trim=20cm 7.3cm 20cm 7.3cm,clip]{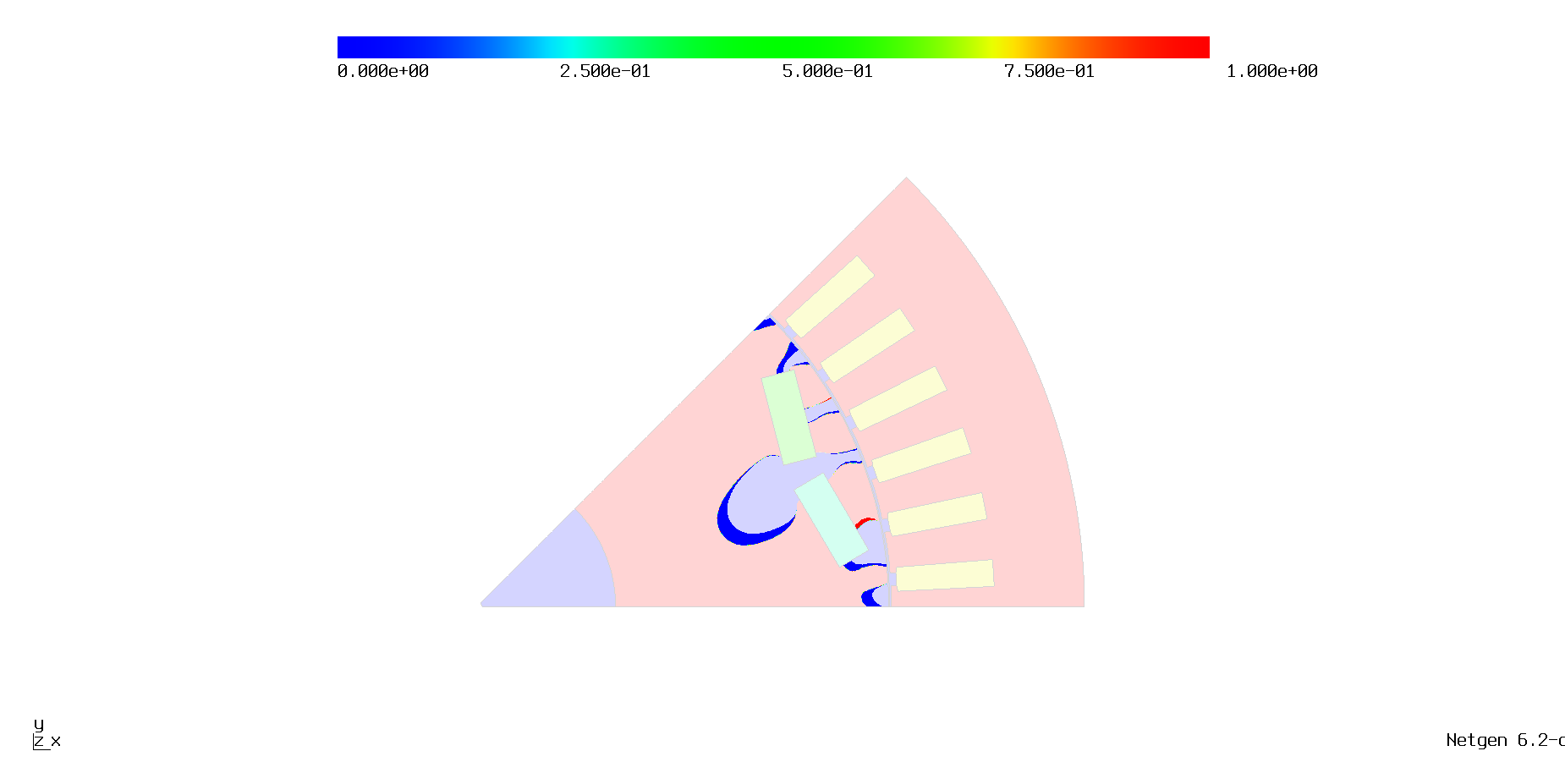}
    \caption{Final design $\Omega_\mathrm{DIST1}$ for nominal optimization (left), difference from the nominal result $\Omega_\mathrm{NOM1}$(right) considering a distributed material uncertainty and a single rotor position, see Section~\ref{ssubsec:DIST1}.}
    \label{fig:design_dist_one}
\end{figure}
\begin{figure}
    \centering
    \includegraphics[width=0.4\textwidth,trim=20cm 7.3cm 20cm 7.3cm,clip]{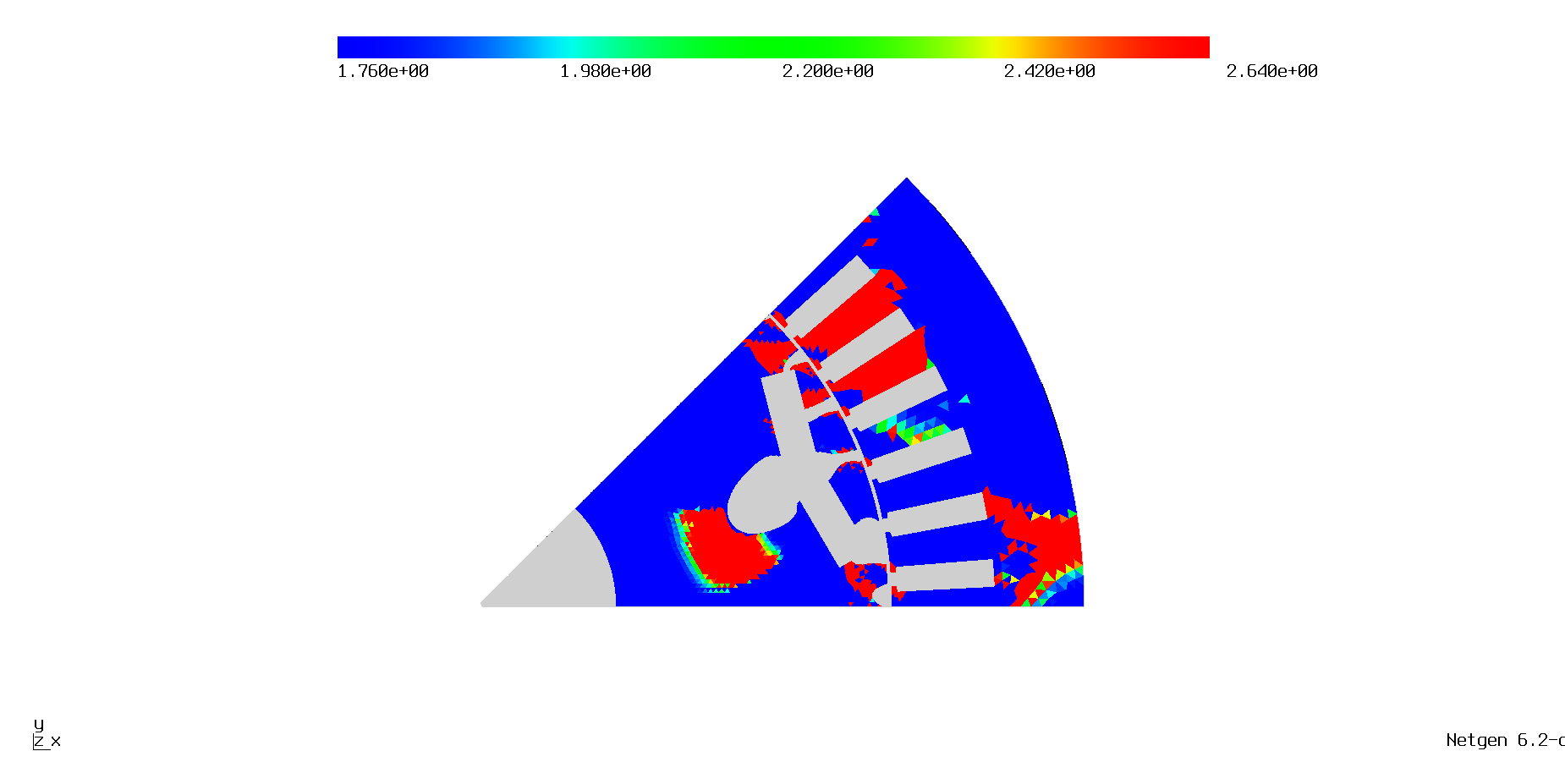}\quad\includegraphics[width=0.4\textwidth,trim=20cm 7.3cm 20cm 7.3cm,clip]{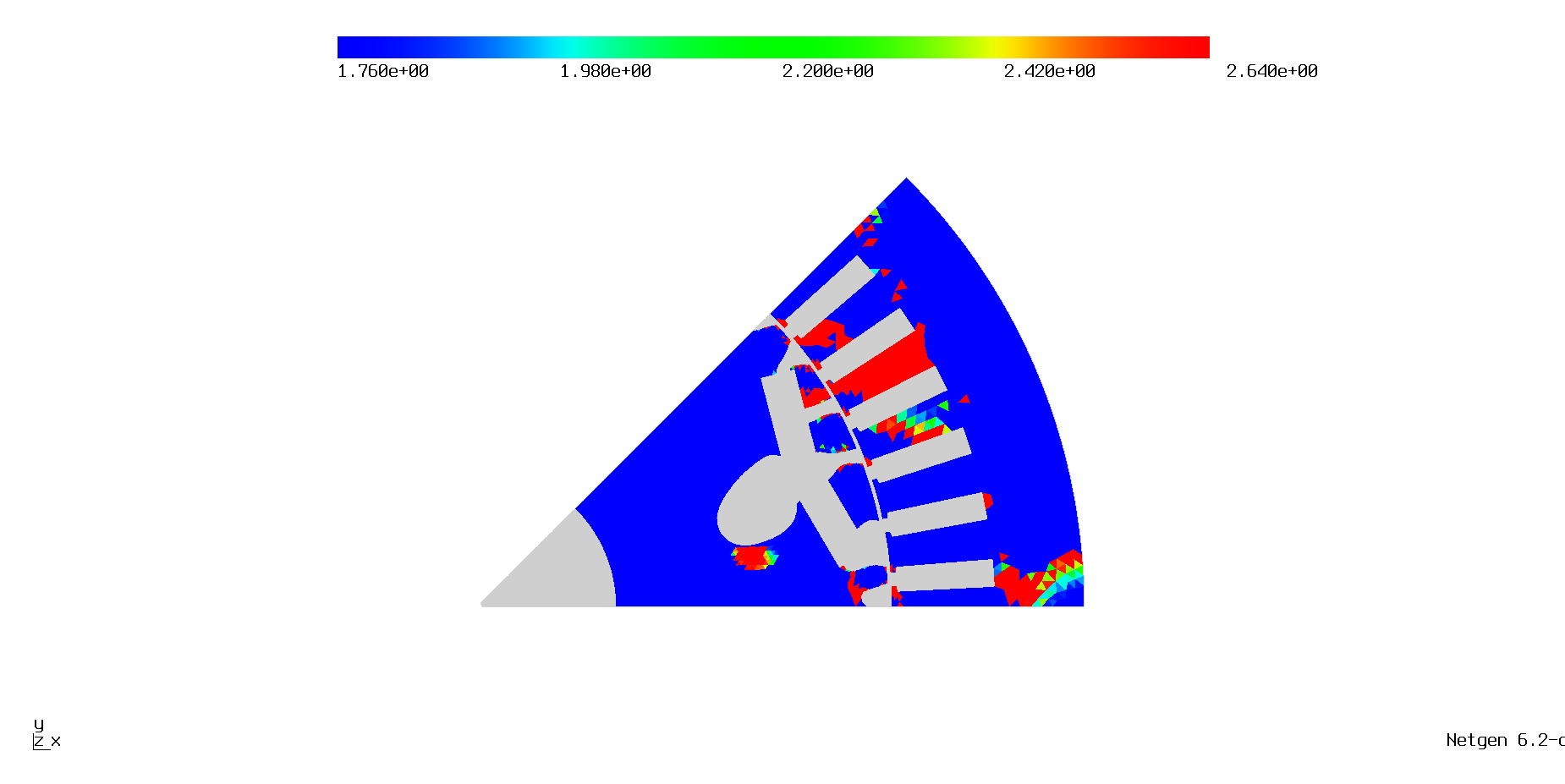}\quad\includegraphics[width=0.103\textwidth]{colorbar.png}
    \caption{Worst case parameter functions for final designs, $q^*(\Omega_\mathrm{NOM1})$ (left), $q^*(\Omega_\mathrm{DIST1})$ (right) considering a distributed material uncertainty and a single rotor position, see Section~\ref{ssubsec:DIST1}.}
    \label{fig:param_dist_one}
\end{figure}
\begin{table}
    \centering
    \begin{tabular}{l|rr}
         Name&Nominal [Nm]&Robust [Nm] \\\hline
         ANG&-757&-778\\
         SCAL&-696&-701\\
         DIST&-692&-696\\
         DIST1&-870&-989
    \end{tabular}
    \caption{Overview of worst case function values for different applications: Uncertain load angle (ANG), see Section~\ref{subsec:ANG}; uncertain scalar material parameter (SCAL), see Section~\ref{ssubsec:SCAL}; uncertain distributed material parameter for average torque (DIST), see Section~\ref{ssubsec:DIST}; single rotor position (DIST1), see Section~\ref{ssubsec:DIST1}. }
    \label{tab:comparison}
\end{table}
\subsection{Computation Time}
All computations were done using a single Intel i5 core with 2.4GHz. We present an overview of the optimization for all examples in Table \ref{tab:time}. We see that for the nominal optimizations we performed one gradient evaluation per iteration. Due to the line search we needed more than one function evaluation per iteration for all cases. The algorithm converged slower in terms of number of iterations for the optimizations where we considered a single rotor position (NOM1, DIST1). This indicates that considering the average torque has a smoothing effect on the design optimization problem. In the last column of Table \ref{tab:time} we present the cost of the robust optimization relative to the nominal optimization. The computational overhead of the robust optimizations comes from the evaluation of the worst case function which includes solving the inner maximization problem. The robust optimizations were between 2.1 and 11.9 times slower than the nominal optimizations, except in the case where we optimized the average torque considering a distributed material uncertainty (DIST). If we account for the fact that due to physical reasons we had to consider more rotor positions we may scale the computation time for this case and result at an overhead factor $26.5/3=8.8$ which is in the range of the other robust optimizations.
\begin{table}
    \centering
    \begin{tabular}{l|rrrrrrr}
         Name&Iterations&\makecell{Function\\ evaluations}&\makecell{Gradient \\ evaluations}&\makecell{Rotor\\ positions}&Time [s]&\makecell{Factor to \\ nominal}\\\hline
         NOM&29&46&29&11&620&1\\
         ANG&76&513&203&11&4628&7.5\\
         SCAL&30&165&71&11&1332&2.1\\
         DIST&23&763&304&33&16428&26.5\\\hline
         NOM1&65&88&65&1&111&1\\
         DIST1&67&1973&796&1&1324&11.9
    \end{tabular}
    \caption{Overview of computational effort for different design optimizations: Nominal optimization (NOM), see Section~\ref{subsec:NOM}, robust optimization with uncertain load angle (ANG), see Section~\ref{subsec:ANG}; uncertain scalar material parameter (SCAL), see Section~\ref{ssubsec:SCAL}; uncertain distributed material parameter for torque (DIST), see Section~\ref{ssubsec:DIST}; nominal optimization for single rotor position (NOM1), see Section~\ref{ssubsec:DIST1}; robust optimization with distributed material uncertainty for single rotor position (DIST1), see Section~\ref{ssubsec:DIST1}.}
    \label{tab:time}
\end{table}
% \begin{table}
%     \centering
%     \begin{tabular}{l|rrrrrrrr}
%          Name&Iterations&\makecell{Function\\ evaluations}&\makecell{Gradient \\ evaluations}&\makecell{Rotor\\ positions}&\makecell{Lin. systems \\ solved}&Time [s]&\makecell{Factor to \\ nominal}&\makecell{Time per\\ solve [s]} \\\hline
%          NOM&29&46&29&11&3155&620&1&0.20\\
%          ANG&76&513&203&11&28085&4628&7.5&0.16\\
%          SCAL&30&165&71&11&8296&1332&2.1&0.16\\
%          DIST&23&763&304&33&94408&16428&26.5&0.17\\\hline
%          NOM1&65&88&65&1&416&111&1&0.27\\
%          DIST1&67&1973&796&1&7687&1324&11.9&0.17
%     \end{tabular}
%     \caption{Comparison of optimization time}
%     \label{tab:time}
% \end{table}

 \section{Conclusion and Future Work}

This paper presented a novel and efficient method for robust design optimization of electric machines, combining the topological derivative and a level set-based approach with a robust optimization framework. The method efficiently solves min-max problems to find designs that are robust against material uncertainties. The proposed approach was applied to a two-material robust topology optimization of a PMSM machine, demonstrating improved worst case performance with minimal trade-offs in nominal performance. Although the parameters load angle and material saturation were used, any parameter can be selected to be robustified. The results highlight the potential of robust topology optimization for designing reliable and efficient electric machines. We have proven theoretical results for linear PDE constraints. However, the method has also turned out to be effective in the more general case of quasilinear PDEs. A theoretical justification for this generalization is the subject of ongoing research. Future work could focus on several promising extensions of this framework. One direction is to incorporate the geometric uncertainty through the level set function, which would represent the manufacturing tolerances due to a production process. In addition, the approach could be applied to other engineering problems with complex design requirements, demonstrating its application across disciplines. Finally, uncertainty quantification techniques could be used to model more realistic and data-driven uncertainty sets, making them more applicable to industry.
\label{sec:Conclusion}
\appendix
\section{Appendix}\label{sec:app}
The following lemma is based on results presented in \citep{Novotny2013} where the topological derivative for a linear diffusion problem is obtained as a limit of the shape derivative with respect to the variation of a circular hole as the hole's radius tends to zero. We use these results to show that Assumption 1 is satisfied in the case of linear material behavior.
\begin{Lemma}\label{lem:app}
    Let $\J:\E\rightarrow\R$ be the reduced cost function of the nominal problem \eqref{eq:reducedPROBLEM} with a linear PDE constraint, i.e. $h_f(b)=\nu_fb$ in \eqref{eq:magnetostatic}. Then Assumption \ref{ass:ass} holds true.
\end{Lemma}
\begin{proof}
We will use the shape derivative according to \citep{Zolesio1992}. 
     A shape function $\J:\E\rightarrow\R$ is said to be shape differentiable if the limit
     \begin{align}
         \id^S\J(\Omega)(V):=\lim_{t\rightarrow 0}\frac{\J((\mathrm{id}+tV)(\Omega))-\J(\Omega)}{t}
     \end{align}
exists for all $V\in C_c^\infty(D)^2$ in $\R$, where $\mathrm{id}$ is the identity map, and the map $V\mapsto \id^S\J(\Omega)(V)$ is linear and continuous with respect to the topology of $C^\infty_c(D)^2$. We call $\id^S\J(\Omega)$ the shape derivative of $\J$.
    For $z\in\Omega,\epsilon>0$ we define a smooth vector field $V^\mathrm{rad}_\epsilon(z)\in C^\infty_c(D)^2$ which coincides on the boundary of the perturbation $\omega_\epsilon(z)$ with the outer unit normal vector
\begin{align}\label{eq:Vrad}
    V^\mathrm{rad}_\varepsilon(z)|_{\partial\omega_\epsilon}=\frac{x-z}{\epsilon}.
\end{align}
As mentioned in Remark \ref{rem:diffusion} we can rewrite the constraining PDE \eqref{eq:magnetostatic} as a linear diffusion equation. For problems of this kind, the identity from \citep[Prop. 1.1]{Novotny2013}
    \begin{align}\label{eq:NOSOidentiy}
    \frac{\id}{\id\Omega}\J(\Omega)(z)=\lim_{\epsilon\searrow 0}\frac{1}{2|\omega|\epsilon}\id^S\J(\Omega_\epsilon(z))(V^\mathrm{rad}_\epsilon(z))
\end{align}
holds, see \citep[Sec. 5.1]{Novotny2013}. Let $\tau_0>0$ be small enough such that $\omega_{\sqrt{\tau_0}}(z)$ does not intersect with other materials, i.e. $\partial\omega_{\sqrt{\tau_0}}(z)\cap(\partial\Omega_f\cup\partial\Omega_a)=\emptyset.$ For $\tau\in(0,\tau_0], \delta\in(-\tau,\tau_0-\tau)$ we interpret $\Omega_{\sqrt{\tau+\delta}}(z)=(\mathrm{id}+(\sqrt{\tau+\delta}-\sqrt{\tau})V_{\sqrt{\tau}}^\mathrm{rad}(z))(\Omega_{\sqrt{\tau}}(z))$ as the deformation of $\Omega_{\sqrt{\tau}}(z)$ by $V^\mathrm{rad}_{\sqrt{\tau}}$ as defined in \eqref{eq:Vrad}.
Using the shape differentiability of $\J$ we conclude for the function $\hat{g}(\tau)=\frac{1}{|\omega|}\J(\Omega_{\sqrt{\tau}}(z))$
\begin{align*}
    \hat{g}'(\tau)&=\lim_{\delta\rightarrow 0}\frac{\J(\Omega_{\sqrt{\tau+\delta}}(z))-\J(\Omega_{\sqrt{\tau}}(z))}{|\omega|\delta}\\&=\lim_{\delta\rightarrow 0}\frac{\J((\mathrm{id}+(\sqrt{\tau+\delta}-\sqrt{\tau})V_{\sqrt{\tau}}^\mathrm{rad}(z))(\Omega_{\sqrt{\tau}}(z)))-\J(\Omega_{\sqrt{\tau}}(z))}{\sqrt{\tau+\delta}-\sqrt{\tau}}\frac{\sqrt{\tau+\delta}-\sqrt{\tau}}{|\omega|\delta}\\&=\lim_{t\rightarrow0}\frac{\J((\mathrm{id}+t V_{\sqrt{\tau}}^\mathrm{rad}(z))(\Omega_{\sqrt{\tau}}(z)))-\J(\Omega_{\sqrt{\tau}}(z))}{t}\lim_{\delta\rightarrow 0}\frac{\sqrt{\tau+\delta}-\sqrt{\tau}}{|\omega|\delta}\\&=\id^S\J(\Omega_{\sqrt{\tau}}(z))(V_{\sqrt{\tau}}^\mathrm{rad}(z))\frac{1}{2|\omega|\sqrt{\tau}},
\end{align*}
which exists since $\tau>0$. Here we used the substitution $t=\sqrt{\tau+\delta}-\sqrt{\tau}$. For $\tau=0$ we have similarly as in Lemma \ref{lem:gtilde}
\begin{align*}
    \hat{g}'(0)=\lim_{\delta\searrow 0}\frac{\hat{g}(\delta)-\hat{g}(0)}{\delta}=\lim_{\epsilon\searrow 0}\frac{\J(\Omega_{\epsilon}(z))-\J(\Omega)}{|\omega|\epsilon^2}=\frac{\id}{\id\Omega}\J(\Omega)(z).
\end{align*}
Finally we show that the derivative is continuous in $0$ by applying \eqref{eq:NOSOidentiy}
\begin{align*}
    \lim_{\tau\searrow 0}\hat{g}'(\tau)&=\lim_{\tau\searrow 0}\id^S\J(\Omega_{\sqrt{\tau}}(z))(V_{\sqrt{\tau}}^\mathrm{rad}(z))\frac{1}{2|\omega|\sqrt{\tau}}\\&=\lim_{\epsilon\searrow 0}\id^S\J(\Omega_{\epsilon}(z))(V_{\epsilon}^\mathrm{rad}(z))\frac{1}{2|\omega|\epsilon}=\frac{\id}{\id\Omega}\J(\Omega)(z)=\hat{g}'(0).
\end{align*}
\end{proof}
\begin{remark}
    The missing step to prove Corollary \ref{cor} for the general case of a quasilinear PDE constraint, which would cover our application, is the identity \eqref{eq:NOSOidentiy}. It is, to our knowledge, still open and subject of ongoing research.
\end{remark}

\subsection*{Funding}
The work of the authors is partially supported by the joint DFG/FWF Collaborative Research Centre CREATOR (DFG: Project-ID 492661287/TRR 361; FWF: 10.55776/F90) at TU Darmstadt, TU Graz, JKU Linz and RICAM Linz.
Further support is given by the Graduate School CE within the Centre for Computational Engineering at TU Darmstadt. P.G. is partially supported by the State of Upper Austria.

\end{document}